\documentclass[a4paper,12pt]{article}

\usepackage{a4wide}
\usepackage{amsmath}
\usepackage{amsfonts,amsthm,amssymb,eucal}
\usepackage{indentfirst}
\usepackage[unicode]{hyperref}

\usepackage{mathtext}

\usepackage[cp1251]{inputenc}
\usepackage[T2A]{fontenc}

\newtheorem{theorem}{Theorem}

\newtheorem{lemma}{Lemma}
\newtheorem{definition}{Definition}
\newtheorem*{remark}{Remark}

\numberwithin{equation}{section}

\newcommand{\Z}{\mathbb{Z}}
\newcommand{\Q}{\mathbb{Q}}
\newcommand{\N}{\mathbb{N}}
\newcommand{\R}{\mathbb{R}}
\newcommand{\A}{\mathbb{A}}
\newcommand{\Comp}{\mathbb{C}}
\newcommand{\Pl}{\mathcal{P}}

\title{On distribution of points with algebraically conjugate coordinates in neighborhood of smooth curves}
\author{V. Bernik, F. G\"otze, A. Gusakova}
\date{}

\begin{document}

\maketitle

\renewcommand{\thefootnote}{}

\footnote{\emph{Key words and phrases}: algebraic numbers, metric theory of Diophantine approximation, Lebesgue measure.}

\footnote{Supported by SFB-701, Bielefeld University (Germany).}

\renewcommand{\thefootnote}{\arabic{footnote}}
\setcounter{footnote}{0}

\begin{abstract}

Let $\varphi:\R\rightarrow \R$ be a continuously differentiable function on an interval $J\subset\R$ and let $\boldsymbol{\alpha}=(\alpha_1,\alpha_2)$ be a point with algebraically conjugate coordinates such that the minimal polynomial $P$ of $\alpha_1,\alpha_2$ is of degree $\leq n$ and height $\leq Q$. Denote by $M^n_\varphi(Q,\gamma, J)$ the set of such points $\boldsymbol{\alpha}$ such that $|\varphi(\alpha_1)-\alpha_2|\leq c_1 Q^{-\gamma}$. We show that for a real $0<\gamma<1$ and any sufficiently large $Q$ there exist positive values  $c_2<c_3$, where $c_i=c_i(n)$, $i=1,2$, which are independent of $Q$, such that $c_2\cdot Q^{n+1-\gamma}<\# M^n_\varphi(Q,\gamma, J)< c_3\cdot Q^{n+1-\gamma}$.

\end{abstract}

\section{Introduction}

First of all let us introduce some useful notation. Let $n$ be a positive integer and $Q>1$ be a sufficiently large real number. Consider a polynomial $P(t)=a_nt^n+\ldots+a_1t+a_0\in\Z[t]$. Denote by $H(P)=\max\limits_{0\leq j \leq n}{|a_j|}$ the height of the polynomial $P$, and by $\deg P$ the degree of the polynomial $P$. We define the following class of integer polynomials with bounded height and degree:
\[
\Pl_{n}(Q):=\{P\in\Z[t]:\deg P\leq n, H(P)\leq Q\}.
\]
Denote by $\#S$ the cardinality of a finite set $S$ and by $\mu_k S$ the Lebesgue measure of a measurable set $S\subset \R^k$, $k\in\N$. Furthermore, denote by $c_j>0$ positive constants independent of $Q$. We are also going to use the Vinogradov symbol $A\ll B$, which means that there exists a constant $c>0$ such that $A\leq c\cdot B$. We will also write $A\asymp B$ if $A\ll B$ and $B\ll A$.

Now let us introduce the concept of an algebraic point. A point $\boldsymbol{\alpha}=(\alpha_1,\alpha_2)$ is called an {\it algebraic point} if $\alpha_1$ and $\alpha_2$ are roots of the same irreducible polynomial $P\in\Z[t]$. The polynomial $P$ of the smallest degree $n\ge 2$ with relatively prime coefficients such that $P(\alpha_1)=P(\alpha_2)=0$ is called the minimal polynomial of algebraic point $\boldsymbol{\alpha}$. Denote by $\deg(\boldsymbol{\alpha})=\deg P$ the degree of the algebraic point $\boldsymbol{\alpha}$ and by $H(\boldsymbol{\alpha})=H(P)$ the height of the algebraic point $\boldsymbol{\alpha}$. Define the following set of algebraic points:
\[
\A_n^2(Q):=\{\boldsymbol{\alpha}\in \Comp^2:\, \deg{\boldsymbol{\alpha}}\leq n,\, H(\boldsymbol{\alpha})\leq Q\}.
\] 
Further denote by $\A_n^2(Q, D):=\A_n^2(Q)\cap D$ the set of algebraic points lying in some domain $D\subset \R^2$.

Problems related to calculating the number of integer points in shapes and bodies in $\R^k$ can be naturally generalized to estimating the number of rational points in domains in Euclidean spaces. Let $f:J_0\rightarrow \R$ be a continuously differentiable function defined on a  finite open interval $J_0$ in $\R$. Define the following set:
\[
N_f(Q,\gamma, J):=\left\{\left(p_1/q,\,p_2/q\right)\in\Q^2:\quad 0<q\leq Q,\, p_1/q\in J,\,\left|f\left(p_1/q\right)-p_2/q\right|<Q^{-\gamma}\right\},
\]
where $J\subset J_0$ and $0\leq\gamma < 2$. In other words, the quantity $\# N_f(Q,\gamma, J)$ denotes the number of rational points with bounded denominators lying within a certain neighborhood of the curve parametrized by $f$. The problem is to estimate the value $\# N_f(Q,\gamma, J)$. In \cite{H96} Huxley proved that for functions $f\in C^2(J)$ such that $0<c_{4}:=\inf\limits_{x\in J_0}|f''(x)|\leq c_{5}:=\sup\limits_{x\in J_0}|f''(x)|< \infty$ and an arbitrary constant $\varepsilon > 0$, the following upper bound holds: 
\[
\# N_f(Q,\gamma,J)\ll Q^{3-\gamma+\varepsilon}.
\]
An estimate without using a quantity  $\varepsilon$ in the exponent has been obtained in 2006 in a paper by Vaughan and Velani \cite{VoVe06}. One year later, Beresnevich, Dickinson and Velani \cite{BeDiVe07} proved a lower estimate of the same order:
\[
\# N_f(Q,\gamma,J)\gg Q^{3-\gamma}.
\]
This result was obtained using methods of metric theory introduced by Schmidt in \cite{Sch80}. 

In this paper we consider a problem related to the distribution of algebraic points $\boldsymbol{\alpha}\in\A_n^2(Q)$ near smooth curves, which is a natural extension of the same problem formulated for rational points. Let $\varphi:J_0\rightarrow \R$ be a continuously differentiable function defined on a finite open interval $J_0$ in $\R$ satisfying the conditions:
\begin{equation}\label{eq1_7}
\sup\limits_{x\in J_0}|\varphi'(x)|:= c_{6} < \infty,\qquad\#\{x\in J_0:\varphi(x)=x\} := c_{7}<\infty.
\end{equation}
Define the following set:
\[
M^n_\varphi(Q,\gamma, J):= \left\{\boldsymbol{\alpha}\in\A_n^2(Q): \alpha_1\in J, |\varphi(\alpha_1)-\alpha_2|<c_1Q^{-\gamma}\right\},
\]
where $c_1=\left(\textstyle\frac12+c_{6}\right)\cdot c_8$ and $J\subset J_0$. This set contains algebraic points with a bounded degree and height lying within some neighborhood of the curve parametrized by $\varphi$. Our goal is to estimate the value $\# M^n_\varphi(Q,\gamma, J)$. The first advancement in solving this problem for $0<\gamma\leq\frac12$ has been made in 2014 in the paper \cite{BeGoKu14}. We are going to state it in the following form: for any $Q>Q_0(n, J, \varphi)$ there exists a positive value $c_{9}>0$ such that $\# M^n_\varphi(Q,\gamma, J) > c_{9}\cdot Q^{n+1-\gamma}$ for $0< \gamma \leq\frac12$.

However, it should be noted that this result is not best possible since for the quantity $\# M^n_\varphi(Q,\gamma, J)$ an upper bound of order $Q^{n+1-\gamma}$ can be proved for $\gamma < 1$. In this paper we are going to fill this gap in the result of \cite{BeGoKu14} by obtaining lower and upper bounds of the same order for $0< \gamma <1$. Our main result is as follows:

\begin{theorem}\label{main}

For any smooth function $\varphi$ with conditions (\ref{eq1_7}) there exist the positive values $c_{2},c_{3}>0$ such that
\[
c_{2}\cdot Q^{n+1-\gamma}<\# M^n_\varphi(Q,\gamma, J) < c_{3}\cdot Q^{n+1-\gamma}
\]
for $Q>Q_0(n, J, \varphi,\gamma)$, sufficiently large $c_1$ and $0< \gamma < 1$.
\end{theorem}

The proof of Theorem \ref{main} is based on the following idea. We consider the strip $L^n_\varphi(Q,\gamma, J):= \left\{\mathbf{x}\in\R^2: x_1\in J, |\varphi(x_1)-x_2|<c_1Q^{-\gamma}\right\}$ and fill it using squares $\Pi=I_1\times I_2$  with sides of length $\mu_1 I_1=\mu_1 I_2 = c_8Q^{-\gamma}$. In order to prove Theorem \ref{main} we need to estimate the number of algebraic points lying in such a square $\Pi$. It should be mentioned that these estimates are highly relevant to several other problems in metric theory of Diophantine approximation \cite{Bug04, BeDo99}.

Let us consider a more general case, namely, the case of a rectangle $\Pi=I_1\times I_2$, where $\mu_1 I_i = c_8Q^{-\gamma_i}$. We are now  going to give an overview of results related to the distribution of algebraic points in rectangles $\Pi$. First of all, let us find the value of the parameter $\gamma_{1}+\gamma_{2}$ such that a rectangles $\Pi$ does not contain algebraic points $\boldsymbol{\alpha}\in\A_n^2(Q)$. The following Theorem \ref{th0} answers this question. The one-dimensional case of this problem was considered in \cite{BeGo15a}.

\begin{theorem}\label{th0}
For any fixed $p,q\in\N$, $p<2q$ there exist rectangles $\Pi_0$ of size $\mu_2\Pi_0=c_{10}(p,q,n)\cdot Q^{-1}$, where
$c_{10}(p,q,n)=\left(2p(2q+2p)^{n}(n+1)\right)^{-1}\cdot q^{n+1}$,
such that $\# \A_n^2(Q,\Pi_0)=0$. 
\end{theorem}

\begin{proof}
Consider the rectangle $\Pi_0$ with sides given by $I_{0,2}=\left(0;\textstyle\frac{p}{q}\right)$ and $I_{0,1}=\left(\textstyle\frac{p}{q};\textstyle\frac{p}{q} + c_{10}\cdot Q^{-1}\right)$. To prove Theorem \ref{th0} assume that there exists an algebraic point $\boldsymbol{\alpha}\in\A_n^2(Q,\Pi_0)$ with the respective minimal polynomial $P_1$. Consider the resultant $R(P_1,P_2)$ of the polynomials $P_1$ and $P_2(t)=qt-p$. Since $\alpha_1\neq \frac{p}{q}$ and $\alpha_2\neq \frac{p}{q}$, we have $|R(P_1,P_2)| > 1$. On the other hand, from Feldman's Lemma (Lemma \ref{lm5}) and the assumption $\boldsymbol{\alpha}\in\Pi_0$ we obtain $|R(P_1,P_2)|<\textstyle\frac12$. This contradiction completes the proof.
\end{proof}

This simple result implies that if the size of the rectangle $\Pi$ is sufficiently large, that is, $\mu_2 \Pi \gg Q^{-1}$, then we have $\# \A_n^2(Q,\Pi)\neq 0$, and we can consider lower bounds for this quantity. A bound of this type was obtained in \cite{BeGoKu14}; it has the form
\begin{equation}\label{eq00}
\# \A_n^2(Q,\Pi) > c_{11}\cdot Q^{n+1}\mu_2\Pi.
\end{equation}

In this paper we obtain an upper bound for $\# \A_n^2(Q,\Pi)$. It is of the same order as the estimate (\ref{eq00}), which demonstrates that the estimate (\ref{eq00}) is asymptotically best possible.

\begin{theorem}\label{th1}
Let $\Pi=I_1\times I_2$ be a rectangle with a midpoint $\mathbf{d}$ and sides $\mu_1 I_i=c_8Q^{-\gamma_{i}}$, $i=1,2$. Then for $0<\gamma_1,\gamma_2< 1$ and $Q>Q_0(n,\gamma, \mathbf{d})$ the estimate
\[
\# \A_n^2(Q,\Pi) < c_{12}\cdot Q^{n+1}\mu_2\Pi
\]
holds, where
\[
c_{12}=2^{3n+9}n^2\rho_n(d_1)\rho_n(d_2)|d_1-d_2|^{-1}\qquad\text{ and }\qquad\rho_n(x)=(\left(|x|+1\right)^{n+1}-1)\cdot |x|^{-1}.
\]
\end{theorem}

It follows from Theorem \ref{th0} that for $1<\gamma_{1}+\gamma_{2}< 2$ we cannot obtain the estimate (\ref{eq00}) for all rectangles $\Pi$. In particular, it is easy to show that certain neighborhoods of algebraic points of small height and small degree do not contain any other algebraic points $\boldsymbol{\alpha}\in\A_n^2(Q)$. This leads us to the definition of a set of small rectangles that are not affected by these ``anomalous'' points. Now let us introduce the concept of a {\it $(v_1,v_2)$-special} square.

\begin{definition}
Let $\Pi=I_1\times I_2$ be a square with midpoint $\mathbf{d}$, $d_1\neq d_2$ and sides $\mu_1 I_1 = \mu_1 I_2 = c_8Q^{-\gamma}$ such that $\frac12< \gamma< 1$. We shall say that the square $\Pi$ satisfies the {\it $(l,v_1,v_2)$-condition} if $v_1+v_2=1$ and there exist not more than $\delta^3\cdot 2^{l+3}Q^{1+2\lambda_{l+1}}\mu_2\Pi$ polynomials $P\in\Pl_2(Q)$ of the form $P(t)=a_2t^2+a_1t+a_0$ satisfying the inequalities  
\[
\begin{cases}
|P(x_{0,i})|<h\cdot Q^{-v_i},\quad i=1,2,\\
\delta Q^{\lambda_{l+1}}\leq|a_2|<\delta Q^{\lambda_l}
\end{cases}
\]
for some point $\mathbf{x}_0\in \Pi$, where $\delta = 2^{-L-17}h^{-2}\cdot (d_1-d_2)^2$, $L=\left[\textstyle\frac{3-2\gamma}{1-\gamma}\right]$ and
\begin{equation}\label{eq2_90}
\lambda_l=
\begin{cases}
1-\frac{(l-1)(1-\gamma)}{2},\quad 1\leq l\leq L+1,\\
\gamma-\frac12,\quad l= L+2,\\
0,\quad l\ge L+3.
\end{cases}
\end{equation}
\end{definition}

\begin{definition}
The square $\Pi=I_1\times I_2$ with sides $\mu_1 I_1 = \mu_1 I_2 = c_8Q^{-\gamma}$ such that $\frac12< \gamma< 1$ is called {\it $(v_1,v_2)$-special} square if it satisfies the {\it $(l,v_1,v_2)$-condition} for all $1\leq l\leq L+2$.
\end{definition}

The following theorem can be proved for  {\it $(v_1,v_2)$-special} squares.

\begin{theorem}\label{th2}
For all {\it $\left(\textstyle\frac12,\textstyle\frac12\right)$-special} squares $\Pi=I_1\times I_2$ with midpoints $\mathbf{d}$, $d_1\neq d_2$ and sides $\mu_1 I_1 = \mu_1 I_2 =c_8Q^{-\gamma}$, where $\frac12< \gamma< 1$ and $c_8>c_0(n,\mathbf{d})$, there exists a value $c_{13}=c_{13}(n,\mathbf{d},\gamma)>0$ such that
\[
\#\A_n^2(Q,\Pi)> c_{13}\cdot Q^{n+1}\mu_2\Pi
\]
for $Q>Q_0(n,\mathbf{d},\gamma)$.
\end{theorem}

\section{Auxiliary statements}

For a polynomial $P$ with roots $\alpha_1,\ldots,\alpha_n$ let $S(\alpha_i) := \left\{x\in\R : |x-\alpha_i| = \min\limits_{1\le j\le n} |x-\alpha_j|\right\}$.
Furthermore, from now on, we assume that the roots of the polynomial $P$ are sorted by distance from $\alpha_i=\alpha_{i,1}$:
\[
|\alpha_{i,1}-\alpha_{i,2}|\le|\alpha_{i,1}-\alpha_{i,3}|\le\ldots\le|\alpha_{i,1}-\alpha_{i,n}|.
\]

\begin{lemma}\label{lm1}
Let $x\in S(\alpha_i)$. Then
\begin{align*}
&|x-\alpha_i| \le n |P(x)|\cdot|P'(x)|^{-1},\quad |x-\alpha_i| \le 2^{n-1} |P(x)|\cdot|P'(\alpha_i)|^{-1},\\
&|x-\alpha_i| \le \min_{1\le j\le n} \left(2^{n-j} |P(x)|\cdot|P'(\alpha_i)|^{-1}\cdot
|\alpha_{i,1}-\alpha_{i,2}|\ldots|\alpha_{i,1}-\alpha_{i,j}| \right)^{1/j}.
\end{align*}
\end{lemma}

The first inequality follows from the inequality $|P'(x)|\cdot |P(x)|^{-1}\leq\sum\limits_{j=1}^n|x-\alpha_{i,j}|^{-1}\leq n|x-\alpha_{i,1}|^{-1}$. For a proof of the second and the third inequalities see \cite{Spr67, Bern83}.

\begin{lemma}[see \cite{Bern80}]\label{lm2}
Let $I\subset\R$ be an interval and let $A\subset I$ be a measurable set,
$\mu_1 A\ge \frac12 \mu_1 I$. If for all $x\in A$ the inequality
$|P(x)|<c_{14}\cdot Q^{-w}$ holds for some $w>0$, then 
\[
|P(x)| < 6^n(n+1)^{n+1}\cdot c_{14}\cdot Q^{-w}
\]
for all points $x\in I$, where $n = \deg P$.
\end{lemma}

\begin{lemma}[see \cite{Per87}]\label{lm3}
Let $\delta$, $\eta_1$, $\eta_2$ be real positive numbers,
and let $P_1, P_2 \in \Z[t]$ be irreducible polynomials of degrees at most  $n$ such that
$\max\left(H(P_1), H(P_2)\right) < K$. Let $J_i\subset \R$, $i=1,2$ be intervals of sizes $\mu_1 J_i=K^{-\eta_i}$. If for some $\tau_1, \tau_2>0$ and for all $\mathbf{x}\in J_1\times J_2$ the inequalities $\max\left(|P_1(x_i)|, |P_2(x_i)|\right) < K^{-\tau_i}$ hold, then 
\begin{equation}\label{eq2_2}
\tau_1+\tau_2+2 + 2\max(\tau_1+1-\eta_1, 0) + 2\max(\tau_2+1-\eta_2, 0) < 2n+\delta
\end{equation}
for $K>K_0(\delta)$.
\end{lemma}

\begin{lemma}[see \cite{Spr67}]\label{lm4}
Let $P\in\Z[t]$ be a reducible polynomial, $P=P_1\cdot P_2$, $\deg P = n\ge 2$. Then 
\[
H(P_1) H(P_2) \asymp H(P).
\]
\end{lemma}

\begin{lemma}[see \cite{Fel51}]\label{lm5}
For any subset of roots $\alpha_{i_1},\ldots,\alpha_{i_s}$, $1\leq s\leq n$, of the polynomial $P(t)=a_nt^n+\ldots+a_1t+a_0$ we have $\prod\limits_{j=1}^{s}|\alpha_{i_j}|\leq (n+1)2^n\cdot H(P)\cdot |a_n|^{-1}$.  
\end{lemma}

\begin{lemma}\label{lm6}
Let $G=G(\mathbf{d},\mathbf{K})$, where $|d_1-d_2|>\varepsilon_1>0$, be a set of points $\mathbf{b}=(b_1,b_0)\in\Z^2$ such that
\begin{equation}\label{eq2_3}
|b_1d_i+b_0|\leq K_i,\quad i=1,2.
\end{equation}
Then
\[
\# G\leq \left(4\varepsilon_1^{-1}K_1+1\right)\cdot\left(4K_2+1\right).
\]
\end{lemma}

\begin{proof}
Without loss of generality we assume $K_1\ge K_2$.
Consider the system of equations
\begin{equation}\label{eq2_4}
b_1d_i+b_0=l_i,\quad i=1,2,
\end{equation}
in two variables. It is clear that for $|l_i|\leq K_i$ any solution of the system (\ref{eq2_4}) satisfies (\ref{eq2_3}). Thus, our problem is reduced to estimating the number of integer solutions of the system (\ref{eq2_4}) with different values $|l_i|\leq K_i$, $i=1,2$.

Let us consider the difference of equations (\ref{eq2_4}): $b_1(d_1-d_2)=l_1-l_2$. Then for $|l_i|\leq K_1$ we obtain:
\[
|b_1|\leq\left(|l_1|+|l_2|\right)\cdot |d_1-d_2|^{-1}\leq 2\varepsilon_1^{-1}K_1.
\]
This inequality implies that all possible values of $b_{1}$ lie in an interval $J_1=\left(-2\varepsilon_1^{-1}K_1, 2\varepsilon_1^{-1}K_1\right)$.

Let us fix the value of $b_1\in J_1$ and consider the system (\ref{eq2_4}) for two different combinations $(b_1, b_{0,0})$ and $(b_1, b_{0,j})$. In this case, the system \ref{eq2_4} can be transformed as follows:
\[
|b_{0,0}-b_{0,j}|=|l_{1,0}-l_{1,j}|\leq 2K_i,\quad i=1,2.
\]
These inequalities imply that for a fixed $b_1$, all possible values of $b_{0}$ lie in an interval $J_0(b_1)=\left(b_{0,0}-2K_2,\, b_{0,0}+2K_2\right)$. Remembering that $b_1,b_0\in\Z$, we have
\[
\# G\leq \left(\mu_1 J_1+1\right)\cdot\left(\mu_1 J_0+1\right)= \left(4\varepsilon_1^{-1}K_1+1\right)\cdot\left(4K_2+1\right).
\]
\end{proof}

\section{Proof of Theorem \ref{th1}}

Assume that $\#\A_n^2(Q,\Pi) \ge c_{12}\cdot Q^{n+1}\mu_2\Pi$. Taking an algebraic point $\boldsymbol{\alpha}\in\A_n^2(Q,\Pi)$ with a minimal polynomial $P$, let us construct an estimate for the polynomial $P$ at points $d_1,d_2$. Since $\alpha_i\in I_i$, we have:
\[
|P^{(k)}(\alpha_{i})|\leq\sum\limits_{j=k}^{n}{\textstyle\frac{j!}{(j-k)!}\cdot|a_{j}|\cdot |\alpha_i|^{j-k}}< \textstyle\frac{n!}{(n-k)!}\cdot \rho_n(d_i)\cdot Q,
\]
for all $1\leq k\leq n$ and $Q>Q_0$. From these estimates and a Taylor expansion of $P$ in the intervals $I_i$, $i=1,2$ we obtain the following inequality:
\begin{equation}\label{eq3_1}
|P(d_i)|\leq\sum\limits_{k=1}^n{\left|\textstyle\frac{1}{k!}P^{(k)}(\alpha_i)(d_i-\alpha_i)^k\right|}< \sum\limits_{k=1}^n{2^{-k}\textstyle{k\choose n}\rho_n(d_i)\cdot Q\mu_1 I_i}\leq 2^n\rho_n(d_i)\cdot Q\mu_1 I_i.
\end{equation}

Let us fix the vector $\mathbf{A}_{1}=(a_n,\ldots,a_{2})$, where $a_n,\ldots,a_{2}$ are the coefficients of the polynomial $P\in \Pl_{n}(Q)$. Denote by $\Pl_{n}(Q,\mathbf{A}_{1})\subset \Pl_n(Q)$ the subclass of polynomials $P$ with the same vector of coefficients $\mathbf{A}_{1}$ such that $P$ satisfies (\ref{eq3_1}). The number of subclasses $\Pl_{n}(Q,\mathbf{A}_{1})$ is equal to the number of vectors $\mathbf{A}_{1}$, which can be estimated as follows for $Q>Q_0$:
\begin{equation}\label{eq3_2}
\#\{\mathbf{A}_{1}\}=(2Q+1)^{n-1}< 2^{n}\cdot Q^{n-1}.
\end{equation}
It should also be noted that every point of the set $\A_n^2(Q,\Pi)$ corresponds to a polynomial $P\in\Pl_n(Q)$ that satisfies (\ref{eq3_1}). On the other hand, every polynomial $P\in\Pl_n(Q)$ satisfying (\ref{eq3_1}) corresponds to no more than $n^2$ points of the set $\A_n^2(Q,\Pi)$. This allows us to write
\[
c_{11}\cdot Q^{n+1}\mu_2\Pi< \# \A_n^2(Q,\Pi)\leq n^2\sum\limits_{\mathbf{A}_{1}}\# \Pl_n(Q,\mathbf{A}_{1}).
\]
Thus, by the estimate (\ref{eq3_2}) and Dirichlet's principle applied to vectors $\mathbf{A}_{1}$ and polynomials $P$ satisfying (\ref{eq3_1}), there exists a vector $\mathbf{A}_{1,0}$ such that
\begin{equation}\label{eq3_2}
\#\Pl_{n}(Q,\mathbf{A}_{1,0})\ge c_{12}\cdot 2^{-n}n^{-2}Q^{2}\mu_2\Pi.
\end{equation}
Let us find an upper bound for the value $\#\Pl_n(Q,\mathbf{A}_{1,0})$. To do this, we fix some polynomial $P_0\in\Pl_{n}(Q,\mathbf{A}_{1,0})$ and consider the difference between the polynomials $P_0$ and $P_j\in\Pl_{n}(Q,\mathbf{A}_{1,0})$ at points $d_i$, $i=1,2$. From the estimate (\ref{eq3_1}) it follows that
\[
|P_0(d_i)-P_j(d_i)|=|(a_{0,1}-a_{j,1})d_i+(a_{0,0}-a_{j,0})|\leq 2^{n+1}\rho_n(d_i)\cdot Q\mu_1 I_i.
\]
Thus the number of different polynomials $P_j\in\Pl_{n}(Q,\mathbf{A}_{1,0})$ does not exceed the number of integer solutions of the following system:
\[
|b_1d_i+b_0|\leq 2^{n+1}\rho_n(d_i)\cdot Q\mu_1 I_i,\quad i=1,2.
\]
Now let us use Lemma \ref{lm6} for $K_i=2^{n+1}\rho_n(d_i)\cdot Q\mu_1 I_i$. Since $\mu_1 I_i = c_8 Q^{-\gamma_i}$ and $\gamma_i < 1$, we have $K_i\ge 2^{n+1}\rho_n(d_i)c_8\cdot Q^{1-\gamma_i}>\max\{\varepsilon_1,1\}$ for $Q>Q_0$. This implies that
\[
j \leq 2^{2n+8}|d_1-d_2|^{-1}\rho_n(d_1)\rho_n(d_2)\cdot Q^2\mu_2\Pi.
\]
It follows, therefore, that $\#\Pl_{n}(Q,\mathbf{A}_{1,0})\leq 2^{2n+8}|d_1-d_2|^{-1}\rho_n(d_1)\rho_n(d_2)\cdot Q^2\mu_2\Pi$, which contradicts inequality (\ref{eq3_2}) for $c_{12}=2^{3n+9}n^2\rho_n(d_1)\rho_n(d_2)|d_1-d_2|^{-1}$. This leads to
\[
\#\A_n^2(Q,\Pi)< c_{12}\cdot Q^{n+1}\mu_2\Pi.
\]

\section{Proof of Theorem \ref{th2}}

\subsection{The main Lemma}

\begin{lemma}\label{lm7}
Let $\Pi=I_1\times I_2$ be a square with midpoint $\mathbf{d}$, $d_1\neq d_2$ and sides $\mu_1 I_1 = \mu_1 I_2 =c_8Q^{-\gamma}$, where $\frac12< \gamma< 1$ and $c_8>c_0(n,\mathbf{d})$. Given positive values $v_1,v_2$ such that $v_1+v_2=n-1$, let $L=L_n(Q,\delta_n,\mathbf{v},\Pi)$ be the set of points $\mathbf{x}\in\Pi$ such that there exists a polynomial  $P\in\Pl_{n}(Q)$ satisfying the following system of inequalities:
\begin{equation}\label{eq4_1}
\begin{cases}
|P(x_i)|< h_n\cdot Q^{-v_i},\\
\min\limits_i\{|P'(x_i)|\}<\delta_n\cdot Q,\quad i=1,2,\\
\end{cases}
\end{equation}
where $h_n = \sqrt{\frac32(|d_1|+|d_2|)\cdot \max\left(1,3|d_1|,3|d_2|\right)^{n^2}}$. If $\Pi$ is a {\it $\left(\frac{v_1}{n-1},\frac{v_2}{n-1}\right)$-special} square, then
\[
\mu_2 L<\textstyle\frac14 \cdot\mu_2\Pi
\]
for $\delta_n<\delta_0(n,\mathbf{d})$ and $Q>Q_0(n,\mathbf{v},\mathbf{d},\gamma)$.
\end{lemma}

\begin{proof}
Since $d_1\neq d_2$ we can assume that for $Q>Q_0$ the following inequality
\begin{equation}\label{eq4_01}
|x_1-x_2|>\varepsilon_1=\textstyle\frac{|d_1-d_2|}{2}
\end{equation}
is satisfied for every point $\mathbf{x}\in\Pi$.

Let us introduce some additional notation. For a polynomial $P$, let $\mathcal{A}(P)$ denote the set of roots of $P$. Denote by $L_1$ and $L_2$ the sets of points $\mathbf{x}\in\Pi$ such that there exists an irreducible polynomial $P\in\Pl_{n}(Q)$ satisfying (\ref{eq4_1}) with a condition $|P'(x_1)|<\delta_nQ$ or $|P'(x_2)|<\delta_nQ$ respectively, and let $L_3$ denote the set of points $\mathbf{x}\in\Pi$ such that (\ref{eq4_1}) is satisfied for some reducible polynomial $P\in\Pl_{n}(Q)$. Clearly, we have $L= L_1\cup L_2\cup L_3$. 

The case of irreducible polynomials will be the most difficult one and requires the largest part of the proof. Let us start by considering this case, deriving estimates for the measures $\mu_1 L_1$ and $\mu_1 L_2$. Without loss of generality, let us assume that $|P'(x_1)|<\delta_nQ$, i.e., consider the set $L_1$.

In this case the main idea is to split an interval $T_i$, which contains all possible values of $P'$ at points $\mathbf{x}\in\Pi$, into sub-intervals $T_{i,1}$, $T_{i,2}$, $T_{i,3}$ and to estimate the measure of the set of solutions of the system (\ref{eq4_1}) for $|P'(x_i)|\in T_{i,k}$, $k=1,2,3$. This splitting is performed as follows:
\[
T_{i,1}=\left[0;\quad 2c_{15}\cdot Q^{\frac12-\frac{v_i}{2}}\right),\quad T_{i,2}=\left[2c_{15}\cdot Q^{\frac12-\frac{v_i}{2}};\quad Q^{\frac{1}{2}-\frac{(n-2)v_i}{2(n-1)}\cdot\theta(n)}\right),\quad i=1,2;
\]
\[
T_{1,3}=\left[Q^{\frac{1}{2}-\frac{(n-2)v_1}{2(n-1)}\cdot\theta(n)};\quad \delta_n\cdot Q\right),\qquad\qquad T_{2,3}=\left[Q^{\frac{1}{2}-\frac{(n-2)v_2}{2(n-1)}\cdot\theta(n)};\quad \rho_{n+1}(d_2)\cdot Q\right),
\]
where $\theta(n)= 0$ if $n\leq 3$ and $\theta(n)= 1$ if $n>3$.

Without loss of generality, let us assume that $|d_1|<|d_2|$. We would like to verify that if a polynomial $P\in\Pl_{n}(Q)$ satisfies the condition 
\begin{equation}\label{eq4_2}
|P'(x_i)|\ge 2c_{15}\cdot Q^{\frac12-\frac{v_i}{2}},
\end{equation}
then the values $|P'(\alpha_i)|$ can be estimated as follows:
\begin{equation}\label{eq4_02}
\textstyle\frac12|P'(x_i)| \leq |P'(\alpha_i)|\leq 2|P'(x_i)|,\quad i=1,2,
\end{equation}
where $x_i\in S(\alpha_i)$ and $c_{15}=2^{n-1}n(n-1)\cdot \max\{h_n,1\}\cdot\max\{1,\rho_{n-1}(d_2)\}$. Let us write a Taylor expansion of $P'$:
\begin{equation}\label{eq4_3}
P'(x_i)=P'(\alpha_i)+P''(\alpha_i)(x_i-\alpha_i)+\ldots+\textstyle\frac{1}{(n-1)!}\cdot P^{(n)}(\alpha_i)(x_i-\alpha_i)^{n-1}.
\end{equation}
Using Lemma \ref{lm1} and estimates (\ref{eq4_1}), (\ref{eq4_2}), we obtain
\[
|x_i-\alpha_i|\leq nh_nc_{15}^{-1}\cdot Q^{-\frac{v_i+1}{2}}<Q^{-\frac{v_i+1}{2}},\qquad  |\alpha_i|\leq |x_i|+\textstyle\frac12<|d_2|+1
\]
for $Q>Q_0$. Let us estimate every term in (\ref{eq4_3}) in the following way:
\[
\left|\textstyle\frac{1}{(k-1)!}\cdot P^{(k)}(\alpha_i)(x_i-\alpha_i)^{k-1}\right|< \textstyle{k-1 \choose n-1}\cdot n(n-1)\rho_{n-1}(d_2)\cdot Q^{\frac12-\frac{v_i}{2}},
\]
for $Q>Q_0$ and $2\leq k\leq n$. Thus, we can write 
\[
\left|\sum_{k=2}^n \textstyle\frac{1}{(k-1)!}\cdot P^{(k)}(\alpha_i)(x_i-\alpha_i)^{k-1}\right|<2^{n-1}n(n-1)\rho_{n-1}(d_2)\cdot Q^{\frac12-\frac{v_i}{2}}<\textstyle\frac12|P'(x_i)|.
\]
Substituting this inequality into (\ref{eq4_3}) yields (\ref{eq4_02}).

This means that for $|P'(x_i)|\in T_{i,3}$ and $|P'(x_i)|\in T_{i,2}$ we have $|P'(\alpha_i)|\in \overline{T}_{i,3}$ and $|P'(\alpha_i)|\in \overline{T}_{i,2}$ respectively, where
\[
\overline{T}_{1,3} =\left[\textstyle\frac12 Q^{\frac{1}{2}-\frac{(n-2)v_1}{2(n-1)}\cdot\theta(n)};\quad 2\delta_n\cdot Q\right), \qquad\qquad \overline{T}_{2,3} =\left[\textstyle\frac12 Q^{\frac{1}{2}-\frac{(n-2)v_2}{2(n-1)}\cdot\theta(n)};\quad 2\rho_{n+1}(d_2)\cdot Q\right),
\]
\[
\overline{T}_{i,2}=\left[c_{15}\cdot Q^{\frac12-\frac{v_i}{2}};\quad 2\cdot Q^{\frac{1}{2}-\frac{(n-2)v_i}{2(n-1)}\cdot\theta(n)}\right),\quad i=1,2.
\]

Let us consider the case $|P'(\alpha_i)|\in \overline{T}_{i,3}$, $i=1,2$. We are going to use induction on the degree of polynomials $P$. 

\subsubsection*{The base of induction: polynomials of second degree.}

Let us consider the system (\ref{eq4_1}) for $n=2$. For a given $u_{2,1},u_{2,2}>0$ under condition $u_{2,1}+u_{2,2}=1$ let $L'=L_2(Q,\delta_2,\mathbf{u}_2,\Pi)$ be the set of points $\mathbf{x}\in\Pi$ such that there exists a polynomial $P\in\Pl_{2}(Q)$ satisfying the system of inequalities
\begin{equation}\label{eq4_4}
\begin{cases}
|P(x_i)|< h_2\cdot Q^{-u_{2,i}},\\
\min\limits_i\{|P'(x_i)|\}<\delta_2\cdot Q,\quad i=1,2.
\end{cases}
\end{equation}
Let us prove that for all {\it $\left(u_{2,1},u_{2,2}\right)$-special} squares $\Pi$ satisfying the conditions of Lemma \ref{lm7}, the estimate
\[
\mu_2 L'<\textstyle\frac{1}{4}\cdot\mu_2\Pi
\]
holds for $\delta_2<\delta_0(\mathbf{d},\gamma)$ and $Q>Q_0(\mathbf{u}_2,\gamma,\mathbf{d})$.

Let $P(t)=a_2t^2+a_1t+a_0$. First, note that the definition of a {\it $(u_{2,1},u_{2,2})$-special} square implies that for $Q>Q_0$ there exists not more than 
\[
\delta 2^{l+3}c_{5}^2Q^{1-2\gamma}<\delta 2^{l+3}c_{5}^2Q^{-\varepsilon}<1 
\]
polynomials $P\in\Pl_2(Q)$ satisfying $|a_2|< \delta Q^{\gamma-\frac12}$ and (\ref{eq4_4}). Therefore, from now on we are going to assume that $|a_2|\ge \delta Q^{\gamma-\frac12}$. 

By the third inequality of Lemma \ref{lm1}, for every polynomial $P$ satisfying the system (\ref{eq4_4}) at a point $\mathbf{x}\in\Pi$ we have the following estimates:
\begin{equation}\label{eq4_5}
|x_i-\alpha_i|<\left(|P(x_i)||a_2|^{-1}\right)^{\frac12} < \delta^{1/2}h_2^{1/2}\cdot Q^{-\frac{2\gamma+2u_{2,i}-1}{4}} < \textstyle\frac{\varepsilon_1}{8},\\
\end{equation}
where $Q>Q_0$ and $x_i\in S(\alpha_i)$, $i=1,2$. Thus, from (\ref{eq4_5}) and (\ref{eq4_01}) we obtain that the distance between the roots $\alpha_1$ and $\alpha_2$ of the polynomial $P$ satisfies
\[
|\alpha_1-\alpha_2|>|x_1-x_2|-|x_1-\alpha_1|-|x_2-\alpha_2|>\textstyle\frac34\cdot\varepsilon_1,
\]
This leads to the following lower bound for $|P'(\alpha_i)|$:
\begin{equation}\label{eq4_6}
|P'(\alpha_i)|=|a_2|\cdot |\alpha_1-\alpha_2|>\textstyle\frac{3}{4}\cdot \varepsilon_1\cdot |a_2|.
\end{equation}
An upper bound for $|P'(\alpha_i)|$ can be obtained from the Taylor expansion of polynomial $P'$:
\[
|P'(\alpha_i)|\leq|P'(x_i)|+|P''(x_i)|\cdot|x_i-\alpha_i|\leq |P'(x_i)|+\textstyle\frac{\varepsilon_1}{4}\cdot|a_2|.
\]
Hence, by (\ref{eq4_6}) and (\ref{eq4_4}) we have
\begin{equation}\label{eq4_9}
|a_2|< 4\varepsilon_1^{-1}\cdot\min\limits_i\{|P'(x_i)|\}<4\delta_2\varepsilon_1^{-1}\cdot Q.
\end{equation}

Now let us turn to the estimation of $\mu_2 L'$. From Lemma \ref{lm1} and the estimates (\ref{eq4_6}) it follows that $L'\subset\bigcup\limits_{P\in\Pl_2(Q)}{\sigma_P}$, where
\begin{equation}\label{eq4_9_1}
\sigma_P=\{\mathbf{x}\in\Pi:\quad |x_i-\alpha_{i}|<2h_2\varepsilon_1^{-1}Q^{-u_{2,i}}|a_2|^{-1},i=1,2\}.
\end{equation}
Simple calculations show that for $c_8 > 2^4h_2\varepsilon_1^{-1}\delta^{-1}$ and $|a_2|>\delta Q^{\gamma-\frac12}$ we have:
\[
\mu_2\sigma_P\leq 2^4h_2^2\varepsilon_1^{-2}Q^{-1}|a_2|^{-2}\leq \textstyle\frac{2^8h_2^2}{\varepsilon_1^2\delta^{2}}\cdot Q^{-2\gamma}<\textstyle\frac14\cdot \mu_2\Pi.
\]

Let $\Pl_{2}(Q,l)\subset\Pl_2(Q)$ be a subclass of polynomials defined as follows: 
\[
\Pl_{2}(Q,l)=\{P\in\Pl_2(Q):\delta Q^{\lambda_{l+1}}\leq |a_2|<\delta Q^{\lambda_l}\},
\]
where $\lambda_l$ is defined by (\ref{eq2_90}) and $\delta = 2^{-L-17}h_2^{-2}\cdot (d_1-d_2)^2$, $L=\left[\textstyle\frac{3-2\gamma}{1-\gamma}\right]$. Thus, by (\ref{eq4_9}) it follows, that for $|a_2|>\delta Q^{\gamma-\frac12}$ and $\delta_2=\frac{4\delta}{\varepsilon_1}$ we have:
\[
\mu_2 L'\leq\mu_2\bigcup\limits_{P\in\Pl_2(Q)}{\sigma_P}\leq
\sum\limits_{l=1}^{L+1}\sum\limits_{P\in\Pl_{2}(Q,l)}{\mu_2\sigma_P}.
\]
From the definition of a {\it $(u_{2,1},u_{2,2})$-special} square it follows, that the number of polynomials $P\in\Pl_{2}(Q,l)$ satisfying (\ref{eq4_4}) does not exceed
\begin{equation}\label{eq4_9_2}
\delta^3\cdot 2^{l+3}Q^{1+2\lambda_{l+1}}\mu_2\Pi.
\end{equation}
Hence, by estimates (\ref{eq4_9_1}) and (\ref{eq4_9_2}) we have:
\[
\mu_2 L_{2}\leq 2^8\varepsilon_1^{-2}h_2^2\delta Q^{-1}\mu_2\Pi\cdot\sum\limits_{l=1}^{L+1}{2^{l+3}Q^{1+2\lambda_{l+1}-2\lambda_{l+1}}}\leq \textstyle\frac{1}{4}\cdot\mu_2\Pi.
\]

\subsubsection*{The induction step: reduction of the degree of the polynomial.}

Let us return to the proof of Lemma \ref{lm7}. For $|P'(\alpha_i)|\in \overline{T}_{i,3}$, $i=1,2$ we consider the following system of inequalities:
\begin{equation}\label{eq4_10}
\begin{cases}
|P(x_i)|< h_n\cdot Q^{-v_i},\quad i=1,2,\\
\frac12 Q^{\frac{1}{2}-\frac{(n-2)v_1}{2(n-1)}\cdot\theta(n)}\leq|P'(\alpha_1)|<2\delta_n\cdot Q,\\
\frac12 Q^{\frac{1}{2}-\frac{(n-2)v_2}{2(n-1)}\cdot\theta(n)}\leq|P'(\alpha_2)|<2\rho_{n+1}(d_2)\cdot Q.
\end{cases}
\end{equation}

Denote by $L_{3,3}$ a set of points $\mathbf{x}\in\Pi$ such that there exists a polynomial $P\in\Pl_{n}(Q)$ satisfying the system (\ref{eq4_10}). By Lemma \ref{lm1}, it follows that $L_{3,3}\subset\bigcup\limits_{P\in\Pl_{n}(Q)}\bigcup\limits_{\boldsymbol{\alpha}\in \mathcal{A}^2(P)}{\sigma_P(\boldsymbol{\alpha})}$, where
\begin{equation}\label{eq4_11}
\sigma_P(\boldsymbol{\alpha}):=\left\{\mathbf{x}\in\Pi:\quad |x_i-\alpha_i|<2^{n-1}h_nQ^{-v_i}|P'(\alpha_i)|^{-1}, i=1,2\right\}.
\end{equation}
It means that the following estimate for $\mu_2 L_{3,3}$ holds:
\[
\mu_2 L_{3,3}\leq\mu_2\bigcup\limits_{P\in\Pl_{n}(Q)}\bigcup\limits_{\boldsymbol{\alpha}\in \mathcal{A}^2(P)}{\sigma_P(\boldsymbol{\alpha})}\leq\sum\limits_{P\in\Pl_{n}(Q)}\sum\limits_{\boldsymbol{\alpha}\in \mathcal{A}^2(P)}{\mu_2\sigma_P(\boldsymbol{\alpha})}.
\]
Together with the sets $\sigma_P(\boldsymbol{\alpha})$ consider the following expanded sets
\begin{equation}\label{eq4_12}
\sigma'_P(\boldsymbol{\alpha})=\sigma'_{P,1}(\alpha_1)\times\sigma'_{P,2}(\alpha_2)=\left\{\mathbf{x}\in\Pi:\quad |x_i-\alpha_i|<c_{16}Q^{-u_{i,n-1}}|P'(\alpha_i)|^{-1}\right\},
\end{equation}
where $u_{i,n-1}=\frac{(n-2)v_i}{n-1}$, $i=1,2$. It is easy to see that the measure of an expanded set $\sigma'_P(\boldsymbol{\alpha})$ is smaller than the measure of the square $\Pi$ for $Q>Q_0$.

Using (\ref{eq4_11}) and (\ref{eq4_12}), we find that the measures of the sets $\sigma_P(\boldsymbol{\alpha})$ and $\sigma'_P(\boldsymbol{\alpha})$ are connected as follows:
\begin{equation}\label{eq4_13}
\mu_2\sigma_P(\boldsymbol{\alpha})\leq 2^{2n-2}h_n^2c_{16}^{-2}\cdot Q^{-1}\mu_2\sigma'_P(\boldsymbol{\alpha}).
\end{equation}

For a fixed $a$, let  $\Pl_{n}(Q,a)\subset \Pl_n(Q)$ denote a subclass of polynomials with the leading coefficient $a$: 
\[
\Pl_{n}(Q,a)=\left\{P\in \Pl_{n}(Q): P(t)=at^n+\ldots+a_0\right\}.
\]
Since $-Q\leq a\leq Q$, the number of subclasses $\Pl_{n}(Q,a)$ is equal to
\begin{equation}\label{eq4_14}
\#\left\{a\right\}= 2Q + 1.
\end{equation}

We are going to use Sprind\v{z}uk's method of essential and non-essential domains \cite{Spr67}. Consider a family of sets $\sigma'_P(\boldsymbol{\alpha})$, $P\in\Pl_{n}(Q,a)$. A set $\sigma'_{P_1}(\boldsymbol{\alpha}_1)$ is called {\it essential} if for every set $\sigma'_{P_2}(\boldsymbol{\alpha}_2)$, $P_2\neq P_1$, the inequality
\[
\mu_2\left(\sigma'_{P_1}(\boldsymbol{\alpha}_1)\cap\sigma'_{P_2}(\boldsymbol{\alpha}_2)\right) < \textstyle\frac12\cdot \mu_2\sigma'_{P_1}(\boldsymbol{\alpha}_1).
\]
is satisfied. Otherwise, the set $\sigma'_{P_1}(\boldsymbol{\alpha}_1)$ is called {\it non-essential}. 

{\it The case of essential sets.} From the definition of essential sets, we immediately have that
\begin{equation}\label{eq4_15}
\sum\limits_{P\in\Pl_{n}(Q,a)}\sum\limits_{\substack{\boldsymbol{\alpha}\in \mathcal{A}^2(P):\\ \sigma'_P(\boldsymbol{\alpha}) \text{---essential}}}{\mu_2\sigma'_P(\boldsymbol{\alpha})}\leq 2\mu_2\Pi.
\end{equation}
Then inequalities (\ref{eq4_13}), (\ref{eq4_14}) and (\ref{eq4_15}) for $c_{16}=2^{n+5}h_n$ allow us to write
\begin{equation}\label{eq4_16}
\sum\limits_{a}\sum\limits_{P\in\Pl_{n}(Q,a)}\sum\limits_{\substack{\boldsymbol{\alpha}\in \mathcal{A}^2(P):\\ \sigma'_P(\boldsymbol{\alpha}) \text{---ess.}}}{\mu_2\sigma_P(\boldsymbol{\alpha})}\leq 2^{-10}\cdot \sum\limits_{P\in\Pl_{n}(Q,a)}\sum\limits_{\substack{\boldsymbol{\alpha}\in \mathcal{A}^2(P):\\ \sigma'_P(\boldsymbol{\alpha}) \text{---ess.}}}{\mu_2\sigma'_P(\boldsymbol{\alpha})} < \textstyle\frac{1}{288}\cdot\mu_2\Pi.
\end{equation}

{\it The case of non-essential sets.} If a set $\sigma'_{P_1}(\boldsymbol{\alpha}_1)$ is {\it non-essential}, then the family contains another set  $\sigma'_{P_2}(\boldsymbol{\alpha}_2)$ such that $\mu_2\left(\sigma'_{P_1}(\boldsymbol{\alpha}_1)\cap\sigma'_{P_2}(\boldsymbol{\alpha}_2)\right) > \textstyle\frac12 \mu_2\sigma'_{P_1}(\boldsymbol{\alpha}_1)$. Consider the difference $R= P_2-P_1$, which is a polynomial of degree $\deg R\leq n-1$ and height $H(R)\leq 2Q$. Let us estimate the polynomials $R$ and $R'$ at points $\mathbf{x}\in\left(\sigma'_{P_1}(\boldsymbol{\alpha}_1)\cap\sigma'_{P_2}(\boldsymbol{\alpha}_2)\right)$.

From the Taylor expansions of the polynomials $P_j$, in the intervals $\sigma'_{P_1,i}(\alpha_{1,i})\cap\sigma'_{P_2,i}(\alpha_{2,i})$, $i,j=1,2$, the estimates (\ref{eq4_10}), (\ref{eq4_12}) and the equality $u_{i,n-1}=\frac{(n-2)v_i}{n-1}$ we have:
\[
|P_j(x_i)|\leq \sum\limits_{k=1}^n{\left|\textstyle\frac{1}{k!}P_j^{(k)}(\alpha_{j,i})(x_i-\alpha_{j,i})^k\right|}\leq \sum\limits_{k=1}^n{\textstyle{k \choose n}\cdot \rho_{n}c_{16}^k\cdot Q^{-u_{i,n-1}}}\leq \rho_n(d_2)(1+c_{16})^n\cdot Q^{-u_{i,n-1}}
\]
for $Q>Q_0$. Now we can write:
\begin{equation}\label{eq4_18_1}
|R(x_i)|<|P_1(x_i)|+|P_2(x_i)|<2\rho_n(d_2)(1+c_{16})^n\cdot Q^{-u_{i,n-1}}. 
\end{equation}

Similarly, Taylor expansions of the polynomials $P'_j$, $j=1,2$ in the intervals $\sigma'_{P_1,i}(\alpha_{1,i})\cap\sigma'_{P_2,i}(\alpha_{2,i})$, the estimates (\ref{eq4_10}), (\ref{eq4_12}) and the equality $u_{i,n-1}=\frac{(n-2)v_i}{n-1}$ allow us to write
\[
|P_j(x_i)|< n^2\rho_{n}(d_2)(1+c_{16})^{n-1}\cdot|P'_j(\alpha_{j,i})|.
\]
From these estimates and the inequalities (\ref{eq4_10}) it easily follows that
\begin{equation}\label{eq4_18_2}
\min\limits_i\left\{|R'(x_i)|\right\}<4n^2\rho_{n}(d_2)(1+c_{16})^{n-1}\delta_n\cdot Q.
\end{equation}

The inequalities (\ref{eq4_18_1}) and (\ref{eq4_18_2}) are satisfied for every point $\mathbf{x}\in \sigma'_{P_1}(\boldsymbol{\alpha}_1)\cap\sigma'_{P_2}(\boldsymbol{\alpha}_2)$. Since $\mu_1\left(\sigma'_{P_1,i}(\alpha_{1,i})\cap\sigma'_{P_2,i}(\alpha_{2,i})\right) > \frac12 \mu_1 \sigma'_{P_1,i}(\alpha_{1,i})$ for $i=1,2$, from Lemma \ref{lm2} it follows that for every point $\mathbf{x}\in \sigma'_{P_1}(\boldsymbol{\alpha}_1)$ the inequalities 

\begin{equation}\label{eq4_18}
|R(x_i)|<c_{17}\cdot Q^{-u_{i,n-1}},\quad \min\limits_i\left\{|R'(x_i)|\right\}<c_{18}\delta_n\cdot Q,
\end{equation}
hold, where $c_{17}=6^n(n+1)^{n+1}\cdot 2\rho_n(d_2)(1+c_{16})^n$ and $c_{18}=6^n(n+1)^{n+1}\cdot 2n^2\rho_{n}(d_2)(1+c_{16})^{n-1}$.

Denote by $L'$ a set of points $\mathbf{x}\in\Pi$ such that there exists a polynomial $R\in\Pl_{n-1}(Q_1)$ satisfying the following system of inequalities:
\[
\begin{cases}
|R(x_i)|< c_{19}h_{n-1}\cdot Q_1^{-u_{i,n-1}},\quad u_{i,n-1} > 0,\\
\min\limits_i\{|R'(x_i)|\}<\delta_{n-1}\cdot Q_1,\\
u_{1,n-1}+u_{2,n-1}=n-2,\quad i=1,2,
\end{cases}
\]
where $Q_1= 2Q$, $c_{19}=\max\limits_i\left\{2^{u_i,n-1}\right\}c_{17}h_{n-1}^{-1}$ and $\delta_{n-1}=2c_{18}\cdot \delta_n$.

It should be mentioned that if polynomial $R(t)=a_1t-a_0$ is linear, then by Lemma \ref{lm1} we obtain:
\[
\left|x_i-\textstyle\frac{a_0}{a_1}\right|\ll Q_1^{-\gamma_{j_2,i}}<\textstyle\frac{\varepsilon}{4},\quad i=1,2
\]
for $Q_1>Q_0$. Hence, we immediately have $|x_1-x_2|<\varepsilon$ which contradicts to condition 2 for polynomial $\Pi$.

The estimates (\ref{eq4_18}) imply that the inclusion 
\[
\bigcup\limits_{P\in\Pl_{n}(Q,a)}\bigcup\limits_{\substack{\boldsymbol{\alpha}\in \mathcal{A}^2(P):\\ \sigma'_P(\boldsymbol{\alpha}) \text{---non-ess.}}}{\sigma'_P(\boldsymbol{\alpha})}\subset L'
\]
is satisfied for all $a$. Thus, by the induction assumption, we obtain that 
\begin{equation}\label{eq4_20}
\sum_{a}\sum\limits_{P\in\Pl_{n}(Q,a)}\sum\limits_{\substack{\boldsymbol{\alpha}\in \mathcal{A}^2(P):\\ \sigma'_P(\boldsymbol{\alpha}) \text{---non-ess.}}}{\mu_2\sigma_P(\boldsymbol{\alpha})}\leq \mu_2 L'\leq \textstyle\frac{1}{288}\cdot\mu_2\Pi,
\end{equation}
for a sufficiently small constant $\delta_n$ and $Q>Q_0$. Then the estimates (\ref{eq4_16}) and (\ref{eq4_20}) allow us to write
\[
\mu_2 L_{3,3}\leq \sum_{a}\sum\limits_{P\in\Pl_{n}(Q,a)}\sum\limits_{\substack{\boldsymbol{\alpha}\in \mathcal{A}^2(P):\\ \sigma'_P(\boldsymbol{\alpha}) \text{---ess.}}}{\mu_2\sigma_P(\boldsymbol{\alpha})}+\sum_{a}\sum\limits_{P\in\Pl_{n}(Q,a)}\sum\limits_{\substack{\boldsymbol{\alpha}\in \mathcal{A}^2(P):\\ \sigma'_P(\boldsymbol{\alpha}) \text{---non-ess.}}}{\mu_2\sigma_P(\boldsymbol{\alpha})}\leq\textstyle\frac{1}{144}\cdot\mu_2\Pi.
\]

\subsubsection*{The case of sub-intervals $T_{1,n}$ and $T_{2,n}$}

For $|P'(\alpha_i)|\in \overline{T}_{i,2}$, $i=1,2$, we have the following system of inequalities:
\begin{equation}\label{eq4_21}
\begin{cases}
|P(x_i)|< h_n\cdot Q^{-v_i},\\
c_{15}\cdot Q^{\frac12-\frac{v_i}{2}} \leq|P'(\alpha_i)|<2Q^{\frac{1}{2}-\frac{(n-2)v_i}{2(n-1)}\cdot \theta(n)},\quad i=1,2.
\end{cases}
\end{equation}

Denote by $L_{2,2}$ a set of points $\mathbf{x}\in\Pi$ such that there exists a polynomial $P\in\Pl_{n}(Q)$ satisfying (\ref{eq4_21}). By Lemma \ref{lm1}, the set $L_{2,2}$ is contained in a union $\bigcup\limits_{P\in\Pl_{n}(Q)}\bigcup\limits_{\boldsymbol{\alpha}\in \mathcal{A}^2(P)}{\sigma_P(\boldsymbol{\alpha})}$, where
\begin{equation}\label{eq4_22}
\sigma_P(\boldsymbol{\alpha})=\left\{\mathbf{x}\in\Pi:|x_i-\alpha_i|<2^{n-1}h_nc_{15}^{-1}Q^{-\frac{v_i+1}{2}},\quad i=1,2\right\}.
\end{equation}

In this case we cannot use induction since the degree of the polynomial cannot be reduced. Let us estimate the measure of the set $L_{2,2}$ by a different method. Without loss of generality, we can assume that $v_1\leq v_2$.

Let us cover the square $\Pi$ by a system of disjoint rectangles $\Pi_k=J_{k,1}\times J_{k,2}$, where $\mu_1 J_{k,i}=Q^{-\frac{v_i+1}{2}+\varepsilon_{2,i}}$, $i=1,2$. The number of rectangles $\Pi_k$ can be estimated as follows:
\begin{equation}\label{eq4_24}
k\leq 4\max\left\{\frac{\mu_1 I_1}{\mu_1 J_{k,1}}, 1\right\}\cdot\max\left\{\frac{\mu_1 I_2}{\mu_1 J_{k,2}}, 1\right\}=
\begin{cases}
4Q^{\frac{n+1}{2}-\varepsilon_{2,1}-\varepsilon_{2,2}}\mu_2\Pi,\quad \gamma<\frac{v_i+1}{2},\\
4Q^{\frac{v_2+1}{2}-\varepsilon_{2,2}}\mu_1 I_2,\quad \gamma\ge\frac{v_1+1}{2}. 
\end{cases}
\end{equation}

We are going to say that a polynomial $P$ belongs to $\Pi_k$ if there is a point $\mathbf{x}\in\Pi_k$ such that the inequalities (\ref{eq4_21}) are satisfied. Let us prove that a rectangle $\Pi_k$ cannot contain two irreducible polynomials $P\in\Pl_{n}(Q)$. Assume the converse: the system of inequalities (\ref{eq4_21}) holds for some irreducible polynomials $P_j$ at some point $\mathbf{x}_j\in\Pi_k$, $j=1,2$. It means that for $Q>Q_0$ and all points $\mathbf{x}\in\Pi_k$ the estimates
\begin{equation}\label{eq4_22}
|x_i-\alpha_{j,i}|\leq|x_i-x_{j,i}|+|x_{j,i}-\alpha_{j,i}|\leq 2\cdot Q^{-\frac{v_i+1}{2}+\varepsilon_{2,i}}<Q^{-\frac{v_i+1}{2}+2\varepsilon_{2,i}}
\end{equation}
are satisfied, where $x_{j,i}\in S(\alpha_{j,i})$.

Let us estimate the absolute values $|P_j(x_i)|$, $i,j=1,2$, where $\mathbf{x}\in\Pi_k$. From the Taylor expansions of $P_j$ in the interval $J_{k,i}$ and estimates (\ref{eq4_21}), (\ref{eq4_22}), we obtain that:
\[
|P_j(x_i)|\leq \rho_n(d_2)3^n\cdot Q^{-v_i+\frac{v_i}{2(n-1)}+(n-1)\varepsilon_{2,i}}<Q^{-v_i+\frac{v_i}{2(n-1)}+n\varepsilon_{2,i}},
\]
for $Q>Q_0$ and $\varepsilon_{2,i}<\textstyle\frac{v_i}{2(n-1)}$.

Applying Lemma \ref{lm3} for $\eta_i = \frac{v_i+1}{2}-2\varepsilon_{2,i}$, $\tau_i=v_i-\frac{v_i}{2(n-1)}-n\cdot \varepsilon_{2,i}$, $i=1,2$ and $\varepsilon_{2,i}=\textstyle\frac{v_i}{8(n-1)}$ leads to the inequality
\[
\tau_1+\tau_2+2+2(\tau_1+1-\eta_1)+2(\tau_2+1-\eta_2)>2n+\textstyle\frac{1}{4}.
\] 
This contradiction shows that there is at most one irreducible polynomial $P\in\Pl_{n}(Q)$ that belongs to a rectangle $\Pi_k$. Hence, by the inequalities (\ref{eq4_22}) and (\ref{eq4_24}) for $Q>Q_0$ we can estimate the measure of the set $L_{2,2}$ as follows:
\[
\mu_2 L_{2,2}\leq\sum\limits_{\Pi_k}{\mu_2\sigma_P(\boldsymbol{\alpha})}\ll Q^{-\varepsilon_{2,2}}\mu_2\Pi<\textstyle\frac{1}{144}\cdot\mu_2\Pi.
\]

\subsubsection*{The case of a small derivative}

Let us discuss a situation where $|P'(x_i)|\in T_{i,1}$, $i=1,2$. In this case we can show that if for some polynomial $P$ and a point $\mathbf{x}\in \Pi$ the inequalities (\ref{eq4_1}) are satisfied for $|P'(x_i)|\in T_{i,1}$, then by Lemma \ref{lm1} we have:
\[
\begin{split}
\left|P''(\alpha_{i})(x_i-\alpha_{i})+\ldots+\textstyle\frac{1}{(n-1)!}\cdot P^{(n)}(\alpha_{i})(x_i-\alpha_{i})^{n-1}\right|<c_{15} Q^{\frac12-\frac{v_1}{2}}.
\end{split}
\]
Using the Taylor expansion of the polynomial $P'$ and this estimate we obtain:
\[
|P'(\alpha_i)|<3c_{15}\cdot Q^{\frac12-\frac{v_i}{2}},
\]
which contradicts our assumption.

Denote by $L_{1,1}$ a set of points $\mathbf{x}\in\Pi$ such that there exists a polynomial $P\in\Pl_{n}(Q)$ satisfying
\begin{equation}\label{eq4_25}
\begin{cases}
|P(x_i)|< h_n\cdot Q^{-v_i},\\
|P'(\alpha_i)|<4c_{15}\cdot Q^{\frac12-\frac{v_i}{2}},\quad i=1,2.
\end{cases}
\end{equation}

The polynomials $P\in \mathcal{P}_{n}(Q)$ are going to be classified according to the distribution of their roots and the size of the leading coefficient. This classification was introduced by Sprin\v{z}uk in \cite{Spr67}.

Let $\varepsilon_3 > 0$ be a sufficiently small constant. For every polynomial $P\in \Pl_{n}(Q)$ of degree $3\leq m\leq n$ we define numbers $\rho_{1,j}$ and $\rho_{2,j}$, $2\le j\le m$, as solutions of the following equations
\[
|\alpha_{1,1}-\alpha_{1,j}| = Q^{-\rho_{1,j}},\qquad|\alpha_{2,1}-\alpha_{2,j}| = Q^{-\rho_{2,j}}.
\]
Let us also define vectors $\mathbf{k}_i=\left(k_{i,2},\ldots,k_{i,m}\right)\in\Z^{m-1}$ as solutions of the inequalities
\[
(k_{i,j}-1)\cdot \varepsilon_3 \le \rho_{i,j} < k_{i,j}\cdot\varepsilon_3 ,\quad i=1,2,\quad j=\overline{2,m}.
\]

Clearly, we have $m(m-1)$ pairs of vectors $\mathbf{k}_1, \mathbf{k}_2$ that correspond to a polynomial $P\in\Pl_{n}(Q)$ of degree $2\leq m\leq n$ depending on the choice of roots $\alpha_{1,1}$ and $\alpha_{1,2}$. Let us define subclasses of polynomials $\Pl_{m}(Q, \mathbf{k}_1, \mathbf{k}_2, u)\subset\Pl_{n}(Q)$ as follows. A polynomial $P$ of degree $2\leq m\leq n$ belongs to a subclass $\Pl_{m}(Q, \mathbf{k}_1, \mathbf{k}_2, u)$, if: 1. the pair of vectors $(\mathbf{k}_1, \mathbf{k}_2)$ correspond to the polynomial $P$ for some pair of roots $\alpha_1,\alpha_2$; 2. the leading coefficient of $P$ is bounded as follows: $Q^{u}\leq|a_m|<Q^{u+\varepsilon_3}$, where $u\in\Z\cdot\varepsilon_3$. 

Let us estimate the number of different subclasses $\Pl_{m}(Q, \mathbf{k}_1, \mathbf{k}_2, u)$. Since $1\leq |a_m|\leq Q$, the following estimate holds for $u$: $0\leq u\leq 1-\varepsilon_3$. On the other hand, we can write $Q\gg|\alpha_{j_1}-\alpha_{j_2}| \gg H(P)^{-m+1} \gg Q^{-m+1}$, where $\alpha_{j_1}, \alpha_{j_2}$ are the roots of a polynomial $P$, which leads to the estimate $-\frac{1}{\varepsilon_3}+1\le k_{i,j} \le \frac{m-1}{\varepsilon_3}$. Thus, an integer vector $\mathbf{k}_i = (k_{i,2},\ldots,k_{i,n})$ can assume at most $\left(m\varepsilon_3^{-1}-1\right)^{m-1}$ values. Now the number of subclasses $\Pl_{m}(Q, \mathbf{k}_1, \mathbf{k}_2, l)$ can be estimated as follows:
\begin{equation}\label{eq4_27}
\# \left\{m,\mathbf{k}_1, \mathbf{k}_2, u\right\}\leq nc_{20}^2\cdot(\varepsilon_3^{-1}+1),
\end{equation}
where $c_{20} = \sum\limits_{i=2}^{n}\left(i\varepsilon_3^{-1}-1\right)^{i-1}$.

Define the values $p_{i,j}$, $i=1,2$ as follows:
\begin{equation}\label{eq4_26}
\begin{cases}
p_{i,j} = (k_{i,j+1} + \ldots + k_{i,m})\cdot \varepsilon_3, \qquad 1\le j \le m-1,\\
p_{i,j}= 0,\qquad j=m.
\end{cases}
\end{equation}

Let us consider polynomials $P$ belonging to the same subclass $\Pl_{m}(Q, \mathbf{k}_1, \mathbf{k}_2, u)$. For these polynomials we can write the following estimates for their derivatives at a root $\alpha_i$:
\begin{equation}\label{eq4_28}
\begin{aligned}
Q^{u-p_{i,1}}\leq |P'(\alpha_i)|&=|a_m|\cdot|\alpha_{i,1}-\alpha_{i,2}|\ldots|\alpha_{i,1}-\alpha_{i,m}|\leq Q^{u-p_{i,1}+m\varepsilon_3},\\
|P^{(j)}(\alpha_i)|&\leq \textstyle\frac{m!}{(m-j)!}\cdot Q^{u-p_{i,j}+m \varepsilon_3}.
\end{aligned}
\end{equation}

Since we are concerned only with polynomials satisfying the system (\ref{eq4_25}), we can assume that the following inequalities hold for at least one value of $l$:
\[
Q^{u-p_{1,i}}\leq |P'(\alpha_i)|<4c_{15}Q^{\frac12-\frac{v_i}{2}},\quad i=1,2.
\]
This condition implies that
\begin{equation}\label{eq4_29}
p_{1,1}>u+\textstyle\frac{v_1-1}{2},\qquad p_{2,1}>u+\textstyle\frac{v_2-1}{2}.
\end{equation}

Now let us estimate the measure of the set $L_{1,1}$. From Lemma \ref{lm1} it follows that $L_{1,1}\subset\bigcup\limits_{m,\mathbf{k}_1, \mathbf{k}_2, u}\bigcup\limits_{P\in\Pl_{m}(Q, \mathbf{k}_1, \mathbf{k}_2, u)}\bigcup\limits_{\boldsymbol{\alpha}\in \mathcal{A}^2(P)}{\sigma_P(\boldsymbol{\alpha})}$, where
\[
\sigma_P(\boldsymbol{\alpha}):=\left\{\mathbf{x}\in\Pi:|x_i-\alpha_i| \le \min\limits_{2\le j\le m} \left(\textstyle\frac{2^{m-j}h_nQ^{-v_i}}{|P'(\alpha_{i,1})|}\cdot|\alpha_{i,1}-\alpha_{i,2}|\ldots|\alpha_{i,1}-\alpha_{i,j}| \right)^{1/j},i=1,2\right\}.
\]
This, together with earlier notation (\ref{eq4_26}) and the estimates (\ref{eq4_28}), yields
\[
\sigma_P(\boldsymbol{\alpha}):=\left\{\mathbf{x}\in\Pi:|x_i-\alpha_i| \le \textstyle\frac12\cdot \min\limits_{2\le j\le m} \left((2^mh_n)^{1/j}\cdot \textstyle Q^{\frac{-u-v_i+p_{i,j}}{j}}\right),i=1,2\right\}.
\]
for $P\in \Pl_{m}(Q, \mathbf{k}_1, \mathbf{k}_2, u)$.

The numbers $j=m_1$ and $j=m_2$ in the formula above provide the best estimates for the roots $\alpha_1$ and $\alpha_2$ respectively if the following inequalities are satisfied:
\begin{equation}\label{eq4_30}
(2^mh_n)^{1/m_i}\cdot Q^{\frac{-u-v_i+p_{i,m_i}}{m_i}}\leq (2^mh_n)^{1/k}\cdot Q^{\frac{-u-v_i+p_{i,k}}{k}},\quad 1\leq k\leq m, i=1,2,
\end{equation}
Then
\begin{equation}\label{eq4_31}
\sigma_P(\boldsymbol{\alpha}):=\left\{\mathbf{x}\in\Pi:|x_i-\alpha_i| < \textstyle\frac12\cdot (2^mh_n)^{1/m_i}\cdot Q^{\frac{-u-v_i+p_{i,m_i}}{m_i}},i=1,2\right\}.
\end{equation}

Let us cover the square $\Pi$ by a system of disjoint rectangles $\Pi_{m_1,m_2}=J_{m_1}\times J_{m_2}$, where $\mu_1 J_{m_i}=Q^{-\frac{u+v_i-p_{i,m_i}}{m_i}+\varepsilon_4}$. The number of rectangles $\Pi_{m_1,m_2}$ can be estimated as follows:
\begin{equation}\label{eq4_32}
\# \Pi_{m_1,m_2}\leq 4Q^{\frac{u+v_1-p_{1,m_1}}{m_1}+\frac{u+v_2-p_{2,m_2}}{m_2}-2\varepsilon_4}\cdot\mu_2\Pi.
\end{equation}

Let us show that a rectangle $\Pi_{m_1,m_2}$ cannot contain two irreducible polynomials belonging to the same subclass $\Pl_{m}(Q, \mathbf{k}_1, \mathbf{k}_2, u)$. Assume the converse: let the inequalities (\ref{eq4_25}) hold for some irreducible polynomial $P_j\in \Pl_{m}(Q, \mathbf{k}_1, \mathbf{k}_2, u)$ and some point $\mathbf{x}_j\in\Pi_{m_1,m_2}$, $j=1,2$. Then for all points $\mathbf{x}\in\Pi_{m_1,m_2}$, we obtain:
\begin{equation}\label{eq4_33}
|x_i-\alpha_{j,i}|\leq|x_i-x_{j,i}|+|x_{j,i}-\alpha_{j,i}|\leq 2\cdot  Q^{-\frac{u+v_i-p_{i,m_i}}{m_i}+\varepsilon_4}<Q^{-\frac{u+v_i-p_{i,m_i}}{m_i}+2\varepsilon_4},
\end{equation}
where $x_{j,i}\in S(\alpha_{j,i})$ and $Q>Q_0$.

Let us estimate $|P_j(x_i)|$, $i,j=1,2$, where $\mathbf{x}\in\Pi_{m_1,m_2}$. From the Taylor expansions of the polynomials $P_j$ in the intervals $J_{m_i}$ and the inequalities (\ref{eq4_28}), (\ref{eq4_33}), (\ref{eq4_30}) for $Q>Q_0$ we obtain that:
\[
|P_j(x_i)|\leq \rho_m(d_2)\cdot 3^m\cdot Q^{-v_1+m\varepsilon_4+m\varepsilon_3}<Q^{-v_1+(m+1)\varepsilon_4+m\varepsilon_3}.
\]

Let us apply Lemma \ref{lm3} for $\eta_i = \frac{u+v_i-p_{i,m_i}}{m_i}-2\varepsilon_4$ and $\tau_i=v_i-(m+1)\varepsilon_4-m\varepsilon_3$, $i=1,2$. Then for $\varepsilon_3=\frac{1}{12m}$ and $\varepsilon_4=\frac{1}{4(3m+1)}$ we have:
\[
\tau_1+\tau_2+2=(n-1)+2-2m\varepsilon_3-2m\varepsilon_4=n+1-\textstyle\frac16-2(m+1)\varepsilon_4,
\]
\[
2(\tau_i+1-\eta_i)= 2v_i+2-\textstyle\frac16-\textstyle\frac{2(u+v_i-p_{i,m_i})}{m_i}-2m\cdot\varepsilon_4.
\]
Let us estimate the expression $2(\tau_i+1-\eta_i)$ by using the inequalities (\ref{eq4_29}):
\[
2(\tau_i+1-\eta_i)\ge
\begin{cases}
v_i+2-u+\textstyle\frac{2p_{i,m_i}}{m}-\textstyle\frac16-2m\varepsilon_4,\quad m_i\ge 2,\\
v_i+1-\textstyle\frac16-2m\varepsilon_4,\quad m_i=1,
\end{cases}
\ge v_i+1-\textstyle\frac16-2m\cdot\varepsilon_4.
\]
Substituting these expressions into (\ref{eq2_2}) yields
\[
\tau_1+\tau_2+2+2(\tau_1+1-\eta_1)+2(\tau_2+1-\eta_2)>2m+\textstyle\frac12,
\]
which is a contradiction.

This means that there is at most one irreducible polynomial $P\in\Pl_{m}(Q, \mathbf{k}_1, \mathbf{k}_2, u)$ belonging to the rectangle $\Pi_{m_1,m_2}$. Now, by inequalities (\ref{eq4_27}) and (\ref{eq4_31}) for $Q>Q_0$, the measure of the set $L_{1,1}$ can be estimated as follows:
\[
\mu_2 L_{1,1}\leq\sum\limits_{m,\mathbf{k}_1, \mathbf{k}_2, u}{\sum\limits_{\Pi_{m_1,m_2}}{\mu_2\sigma_P}}\ll Q^{-2\varepsilon_4}\cdot \mu_2\Pi<\textstyle\frac{1}{72}\cdot\mu_2\Pi.
\]

\subsubsection*{Mixed cases}

All mixed cases have the same structure and can be proved using Lemma \ref{lm3} and the ideas described above \cite{BeGoGu16}.

Thus, we have $L_1\subset\bigcup\limits_{1\leq i,j\leq 3}{L_{i,j}}$, which leads to the following estimate: 
\[
\mu_2 L_1 \leq \sum\limits_{1\leq i,j\leq 3}{\mu_2 L_{i,j}}\leq 9\cdot\textstyle\frac{1}{144}\cdot\mu_2\Pi=\textstyle\frac{1}{16}\cdot\mu_2\Pi.
\]
Similarly, $\mu_2 L_2\leq \frac{1}{16}\cdot\mu_2\Pi$. These estimates conclude the proof of Lemma \ref{lm7} in the case of irreducible polynomials.

\subsubsection*{The case of reducible polynomials}

In this section we are going to estimate the measure of the set $L_3$. Clearly, the results of Lemma \ref{lm3} do not apply directly to this case. Let a polynomial $P$ of degree $n$ be a product of several (not necessarily different) irreducible polynomials $P_1, P_2, \ldots, P_s$, $s\ge 2$, where $\deg P_i = n_i\ge 1$ and $n_1+\ldots+n_s=n$. Then by Lemma \ref{lm4} we have:
\[
H(P_1)\cdot H(P_2)\cdot\ldots\cdot H(P_s) \leq c_{21} H(P) \leq c_{19} Q.
\]
On the other hand, by the definition of height, we have $H(P_i)\ge 1$, and thus $H(P_i)\leq c_{21} Q$, $i=1,\ldots,s$.

Denote by $L_{3}(k,\varepsilon_5)$ a set of points $\mathbf{x}\in\Pi$ such that there exists a polynomial $P\in\Pl_{k}(Q_1)$ satisfying the inequality:
\begin{equation}\label{eq4_49}
|R(x_1)R(x_2)|< h_n^2Q_1^{-k+\varepsilon_5}.
\end{equation}
If a polynomial $P$ satisfies the inequalities (\ref{eq4_1}) at a point $\mathbf{x}\in\Pi$, we can write
\[
|P(x_1)P(x_2)|=|P_1(x_1)P_1(x_2)|\cdot\ldots\cdot|P_s(x_1)P_s(x_2)|\leq h_n^2Q^{-n+1}.
\]
Since $n= n_1 +\ldots+ n_s$ and $s\ge 2$, it is easy to see that at least one of the inequalities
\begin{align}\label{eq4_41}
|P_i(x_1)P_i(x_2)| &\leq h_n^2Q^{-n_i+\gamma},\quad n_i\ge 2,\\
|P_i(x_1)P_i(x_2)| &\leq h_n^2Q^{-\gamma},\quad n_i= 1, i=1,\ldots, s, \notag
\end{align}
is satisfied at the point $\mathbf{x}$. Hence, $\mathbf{x}\in L_3\left(n_j,\gamma\right)$ for $n_j \ge 2$ or $\mathbf{x}\in L_3\left(1,1-\gamma\right)$ and we have 
\[
L_3\subset\left(\bigcup\limits_{k=2}^{n-1}L_{3}(k,\gamma)\right)\cup L_3(1,1-\gamma).
\]

Let us estimate the measure of the set $L_{3}(k,\gamma)$, $2\leq k\leq n-1$. Denote by $L_{3}^1(k,t)$ a set of points $(x_1,x_2)\in\Pi$ such that there exists a polynomial $P\in\Pl_{k}(Q_1)$ satisfying the inequalities:
\begin{equation}\label{eq4_40}
\begin{cases}
|P(x_1)|<h_{n}^2Q_1^t,\quad |P(x_2)|<h_{n}^2Q_1^{-k+1-t},\\
\min\limits_i\left\{|P'(\alpha_i)|\right\}<\delta_k Q_1,\quad x_i\in S(\alpha_i),i=1,2.
\end{cases}
\end{equation}
and by $L_{3}^2(k,t)$ a set of points $(x_1,x_2)\in\Pi$ such that there exists a polynomial $P\in\Pl_{k}(Q_1)$ satisfying the inequality:
\begin{equation}\label{eq4_41}
\begin{cases}
|P(x_1)|<h_{n}^2Q_1^t,\quad |P(x_2)|<h_{n}^2Q_1^{-k+\frac{1+\gamma}{2}-t},\\
|P'(\alpha_i)|>\delta_k Q_1,\quad x_i\in S(\alpha_i),\quad i=1,2.
\end{cases}
\end{equation}
By the definition of the set $L_3(k,\gamma)$ it is easy to see that:
\[
L_3(k,\gamma)\subset \left(\bigcup\limits_{i=0}^{N_1}{L_3^1(k, 1-i(1-\gamma))}\right)\cup\left(\bigcup\limits_{i=0}^{N_2}{L_3^2(k, 1-i(1-3\gamma)/2)}\right), 
\]
where $N_1=\left[\frac{2+k-\gamma}{1-\gamma}\right]$ and $N_2=\left[\frac{4+2k-2\gamma}{1-3\gamma}\right]$.

The system (\ref{eq4_40}) is a system of the form (\ref{eq4_1}). Furthermore, as the polynomials $P\in \Pl_{k}(Q_1)$ are irreducible and $k<n$, we can apply the above arguments for a sufficiently small constant $\delta_k$ and $Q_1 > Q_0$ to obtain the following estimate:
\begin{equation}\label{eq4_43}
\mu_2 L_3^1(k,t) < \textstyle\frac{1}{24n(N_1+1)}\cdot \mu_2\Pi.
\end{equation}

Now let us estimate the measure of the set $L_3^2(k,t)$. From Lemma \ref{lm1} it follows that $L_{3}^2(k,t)$ is contained in a union $\bigcup\limits_{P\in\Pl_{k}(Q)}\bigcup\limits_{\boldsymbol{\alpha}\in \mathcal{A}^2(P)}{\sigma_P(\boldsymbol{\alpha},t)}$, where
\[
\sigma_{P}(\boldsymbol{\alpha}, t) := 
\left\{
\mathbf{x}\in\Pi:
\begin{array}{ll}
|x_1-\alpha_1| \leq 2^{k-1}h_n^2\cdot Q^t\cdot|P'(\alpha_1)|^{-1},\\
|x_2-\alpha_2| \leq 2^{k-1}h_n^2\cdot Q^{-k+\frac{1+\gamma}{2}-t}\cdot|P'(\alpha_2)|^{-1}.
\end{array}
\right\}
\]
Let us estimate the value of the polynomial $P$ at a central point $\mathbf{d}$ of the square $\Pi$. A Taylor expansion of the polynomial $P$ can be written as follows:
\begin{equation}\label{eq4_44}
P(d_i)=P'(\alpha_i)(d_i-\alpha_i)+\textstyle\frac12 P''(\alpha_i)(d_i-\alpha_i)^2+\ldots +\textstyle\frac{1}{k!}\cdot P^{(k)}(\alpha_i)(d_i-\alpha_i)^k.
\end{equation}
If polynomial $P$ satisfy (\ref{eq4_41}) it follows that:
\begin{equation}\label{eq4_45}
\begin{split}
|d_1-\alpha_1|&\leq |d_1-x_{0,1}|+|x_{0,1}-\alpha_1|\leq \mu_1 I_1 + 2^{k-1}h_n^2\delta_k^{-1}\cdot Q_1^{t-1},\\
|d_2-\alpha_2|&\leq |d_2-x_{0,2}|+|x_{0,2}-\alpha_2|\leq \mu_1 I_2 + 2^{k-1}h_n^2\delta_k^{-1}\cdot Q_1^{-k+\frac{1+\gamma}{2}-t-1}.
\end{split}
\end{equation}
Without loss of generality, let us assume that $t\ge -k+\frac{1+\gamma}{2}-t$. Then we can rewrite the estimates (\ref{eq4_45}) as follows:
\[
|d_1-\alpha_1|\leq 
\begin{cases}
c_{21}\cdot\mu_1 I_1,\quad t < 1-\gamma,\\
c_{21}\cdot Q_1^{t-1},\quad 1-\gamma \leq t \leq 1,
\end{cases} \qquad |d_2-\alpha_2|\leq \mu_1 I_2.
\]
where $c_{21}=2^{k-1}h_n^2\delta_k^{-1}+ c_8$.

Using these inequalities and expression (\ref{eq4_44}) allows us to write
\begin{equation}\label{eq4_47}
|P(d_1)|<
\begin{cases}
c_{22}\cdot  Q_1\cdot \mu_1 I_1,\quad t < 1-\gamma,\\
c_{22}\cdot  Q_1^t,\quad 1-\gamma \leq t < 1,
\end{cases}\qquad  |P(d_2)|< c_{22} \cdot Q_1\cdot \mu_1 I_2.
\end{equation}

Fix a vector $\mathbf{A}_{1}=(a_k,\ldots,a_{2})$, where $a_k,\ldots,a_{2}$ will denote the coefficients of the polynomial $P\in \Pl_{k}(Q_1)$.
 Consider a subclass $\Pl_{k}(\mathbf{A}_{1})$ of polynomials $P$ which satisfy (\ref{eq4_41}) and have the same vector of coefficients $\mathbf{A}_{1}$. For $Q_1>Q_0$, the number of such classes can be estimated as follows
\begin{equation}\label{eq4_48}
\#\{\mathbf{A}_{1}\}=(2Q_1+1)^{k-1}< 2^{k} Q_1^{k-1}.
\end{equation}
Let us estimate the value $\#\Pl_{k}(\mathbf{A}_{1})$. Take a polynomial $P_0\in\Pl_{k}(\mathbf{A}_{1})$ and consider the difference between the polynomials $P_0$ and $P_j\in\Pl_{k}(\mathbf{A}_{1})$ at points $d_i$, $i=1,2$. By (\ref{eq4_47}), we have that:
\[
|P_0(d_1)-P_j(d_1)|=|(a_{0,1}-a_{j,1})d_1+(a_{0,0}-a_{j,0})|\leq \begin{cases}
2c_{22}\cdot  Q_1\mu_1 I_1,\quad t < 1-\gamma,\\
2c_{22}\cdot  Q_1^t,\quad 1-\gamma \leq t \leq 1,
\end{cases}
\]
\[
|P_0(d_2)-P_j(d_2)|=|(a_{0,1}-a_{j,1})d_2+(a_{0,0}-a_{j,0})|\leq 2c_{22}\cdot  Q_1\mu_1 I_2.
\]
This implies that the number of different polynomials $P_j\in\Pl_{k}(\mathbf{A}_{1})$ does not exceed the number of integer solutions to the system
\[
|b_1d_i+b_0|\leq K_i,\quad i=1,2,
\]
where $K_2=2c_{22}\cdot  Q_1\mu_1 I_2$ and $K_1=
2c_{22}\cdot  Q_1\mu_1 I_1$ if $t < 1-\gamma$ and $K_1=2c_{22}\cdot  Q_1^t$ if $1-\gamma \leq t \leq 1$.

It is easy to see that $K_i\ge 2c_{22}\cdot Q_1^{1-\gamma}>Q_1^{\varepsilon}$ for $Q_1>Q_0$. Thus, by Lemma \ref{lm6}, we have
\[
\#\Pl_{k}(\mathbf{A}_{1})\leq
\begin{cases}
2^7\varepsilon_1^{-1}\cdot  Q_1^{2}\cdot \mu_2\Pi,\quad t < 1-\gamma,\\
2^7\varepsilon_1^{-1}\cdot  Q_1^{t+1}\cdot \mu_1 I_2,\quad 1-\gamma \leq t \leq 1.
\end{cases}
\]

This estimate and the inequality (\ref{eq4_48}) mean that the number $N$ of polynomials $P\in \Pl_{k}(Q_1)$ satisfying the system (\ref{eq4_41}) can be estimated as follows:
\begin{equation}\label{eq4_51}
N\leq
\begin{cases}
2^{k+7}\varepsilon_1^{-1}\cdot  Q_1^{k+1}\cdot \mu_2\Pi,\quad t < 1-\gamma,\\
2^{k+7}\varepsilon_1^{-1}\cdot  Q_1^{k+t}\cdot \mu_1 I_2,\quad 1-\gamma \leq t \leq 1.
\end{cases}
\end{equation}
On the other hand, the measure of the set $\sigma_{P}(\boldsymbol{\alpha},t)$ satisfies the inequality
\begin{equation}\label{eq4_52}
\mu_2\sigma_{P}(\boldsymbol{\alpha},t)\leq
\begin{cases}
2^{2k}h_n^4\delta_k^{-2}\cdot  Q_1^{-k-2+\frac{1+\gamma}{2}},\quad t < 1-\gamma,\\
2^{2k}h_n^4\delta_k^{-2}\cdot  Q_1^{-k-1-t+\frac{1+\gamma}{2}}\cdot \mu_1 I_1,\quad 1-\gamma \leq t \leq 1.
\end{cases}
\end{equation}

Then, by estimates (\ref{eq4_51}) and (\ref{eq4_52}), for $Q_1>Q_0$ we can write
\begin{equation}\label{eq4_50}
\mu_2 L_3^2(k,t)\leq 2^{3k+7}\delta_k^{-2}h_n^4\varepsilon_1^{-1} Q_1^{-\frac{1-\gamma}{2}}\mu_2\Pi <\textstyle\frac{1}{24n(N_2+1)}\cdot\mu_2\Pi.
\end{equation}

The inequalities (\ref{eq4_43}) and (\ref{eq4_50}) lead to the following estimate for the measure of the set $L_{3}(k)$, $2\leq l\leq n-1$:
\[
\mu_2 L_{3}(k,\gamma) \leq \sum\limits_{i=0}^{N_1}\mu_2L_3^1(k, 1-i(1-\gamma)) + 
\sum\limits_{i=0}^{N_2}\mu_2L_3^2(k, 1-i(1-3\gamma)/2)\leq \textstyle\frac{1}{12n}\cdot \mu_2\Pi.
\]

Now let us estimate the measure of the set $L_3(1,1-\gamma)$. For every point $\mathbf{x}\in L_3(1,1-\gamma)$ there exists a rational point $\frac{a_0}{a_1}$ such that
\[
\left|x_1-\textstyle\frac{a_0}{a_1}\right|\cdot\left|x_2-\textstyle\frac{a_0}{a_1}\right|< h_n^2Q_1^{-\gamma}|a_1|^{-2}.
\]
Since $|x_1-x_2|>\varepsilon_1$ one of the values $\left|x_i-\textstyle\frac{a_0}{a_1}\right|$, $i=1,2$ is bigger than $\frac{\varepsilon_1}{2}$. Thus we consider the sets
\begin{equation}\label{eq4_46}
\sigma_i\left(a_0/a_1\right):= 
\left\{
\mathbf{x}\in\Pi: \left|x_i-\textstyle\frac{a_0}{a_1}\right|\leq 2h_n^2\varepsilon_1^{-1}Q_1^{-\gamma}|a_1|^{-2}\right\},\quad i=1,2.
\end{equation}
Simple calculations show that for $c_8>4h_n^2\varepsilon_1^{-1}$ we have
\[
\mu_2 \sigma_i\left(a_0/a_1\right)\leq 4h_n^2\varepsilon_1^{-1}c_8Q_1^{-2\gamma}\leq \mu_2\Pi.
\]
Let us define the following sets
\[
\sigma_i=\bigcup\limits_{1\leq a_0,a_1\leq Q_1}\sigma_i\left(a_0/a_1\right),\quad i=1,2.
\]
It is easy to see that $L_3(1,1-\gamma)\subset \sigma_1\cup\sigma_2$ and we need to estimate the measure of the sets $\sigma_1$ and $\sigma_2$.

For a fixed value $a_1$ let us consider the set $N(a_1):=\left\{a_0\in\Z:\sigma_i\left(a_0/a_1\right)\neq\varnothing\right\}$. The cardinality of this set can be estimated by the following way:
\[
\#N(a_1)\leq\begin{cases}
3\mu_1 I_i\cdot |a_1|^{-1},\quad \left(\mu_1 I_i\right)^{-1}\leq |a_1|\leq Q_1,\\
2, \quad 1\leq |a_1|\leq \left(\mu_1 I_i\right)^{-1}.
\end{cases}
\]
These inequalities together with (\ref{eq4_46}) imply:
\begin{multline*}
\mu_2\sigma_i\leq\sum\limits_{1\leq |a_1|\leq Q_1}{N(a_1)\cdot\mu_2\sigma_i\left(a_0/a_1\right)}\leq 8h_n^2c_8\varepsilon_1^{-1}Q_1^{-2\gamma}\sum\limits_{1\leq |a_1|\leq \left(\mu_1 I_i\right)^{-1}}|a_1|^{-2} +\\
+ 12h_n^2\varepsilon_1^{-1}Q_1^{-\gamma}\mu_2\Pi\sum\limits_{\left(\mu_1 I_i\right)^{-1}\leq |a_1|\leq Q_1}|a_1|^{-1}\leq 2\pi^2c_8h_n^2\varepsilon_1^{-1}Q_1^{-2\gamma}+ \\
+ 12h_n^2\varepsilon_1^{-1}Q_1^{-\gamma}\ln Q\mu_2\Pi\leq \textstyle\frac{1}{24n}\cdot \mu_2\Pi
\end{multline*}
for $Q_1>Q_0$ and $c_8>96n\pi^2h_n^2\varepsilon_1^{-1}$.
Then, 
\[
\mu_2 L_3(1,1-\gamma)\leq \textstyle\frac{1}{12n}\cdot \mu_2\Pi,
\]
and, finally,
\[
\mu_2 L_3\leq \sum\limits_{k=2}^{n-1}\mu_2 L_{3}(k,\gamma)+\mu_2 L_3(1,1-\gamma)\leq \textstyle\frac{n-1}{12n}\cdot\mu_2\Pi\leq \textstyle\frac{1}{12}\cdot\mu_2\Pi.
\]
This proves Lemma \ref{lm7} in the case of reducible polynomials.

Combining the obtained estimates for the different cases yields the final estimate
\[
\mu_2 L\leq \mu_2 L_1 +\mu_2 L_2 +\mu_2 L_3 \leq \textstyle\frac14\cdot\mu_2\Pi.
\]
\end{proof}

\begin{remark}
Note, that in case of reducible polynomials we do not use the inequality $\min\limits_i\left\{|P'(x_i)|\right\}<\delta_nQ$. It means, that the set $L_3$ is the set of points $\mathbf{x}\in\Pi$ such that there exists a reducible polynomial $P\in\Pl_n(Q)$ satisfying the inequalities
\[
|P(x_i)|<h_nQ^{-v_i},\quad i=1,2.
\]
\end{remark}

\subsection{The final part of the proof}

Let us use Lemma \ref{lm7} to conclude the proof. Consider a set $B_1 = \Pi \setminus L_n(Q,\delta_n,\mathbf{v},\Pi)$ for $n\ge 2$, $v_1=v_2 = \frac{n-1}{2}$, $Q>Q_0$ and a sufficiently small constant $\delta_n$. From Lemma \ref{lm7} it follows that
\begin{equation}\label{eq5_1}
\mu_2 B_1 \ge \textstyle\frac34\cdot \mu_2\Pi.
\end{equation}
Now we prove that for every point $\mathbf{x}\in \Pi$ there exists a polynomial $P\in \Pl_n(Q)$ such that 
\[
|P(x_i)|\leq h_n\cdot Q^{-\frac{n-1}{2}},\quad i=1,2.
\]
By Minkowski's linear forms theorem \cite{Sch80} for every point $\mathbf{x}\in \Pi$ there exists a non-zero polynomial $P(t)=a_nt^n+\ldots+a_1t+ a_0\in\Z[t]$ satisfying
\[
|P(x_i)|\leq h_n\cdot Q^{-\frac{n-1}{2}},\quad |a_j|\leq \max\left(1,3|d_1|,3|d_2|\right)^{-n-1}\cdot Q \quad (i=1,2,\quad 2\leq j\leq n).
\]
One can easily verify that $|a_1|<Q$ and $|a_0|<Q$, hence $P\in \Pl_n(Q)$.

Then, by the remark of Lemma \ref{lm7} we can say that for every point $\mathbf{x}_1 \in B_1$ there exists an irreducible polynomial $P_1\in\mathcal{P}_{n}(Q)$ such that
\[
\begin{cases}
|P_1(x_{1,i})|<h_n\cdot Q^{-\frac{n-1}{2}},\\
|P_1'(x_{1,i})|>\delta_n\cdot Q,\quad i=1,2.
\end{cases}
\]
Let us consider the roots $\alpha_1,\alpha_2$ of the polynomial $P_1$ such that $x_{1,i}\in S(\alpha_i)$. By Lemma \ref{lm1}, we have
\begin{equation}\label{eq5_2}
|x_{1,i}-\alpha_i|\leq nh_n\delta_n^{-1}Q^{-\frac{n+1}{2}},\quad i=1,2. 
\end{equation}
Let us prove that $\alpha_1,\alpha_2 \in \R$. Assume the converse: let $\alpha_i\in\Comp$, then the number $\overline{\alpha_i}$ complex conjugate to $\alpha_i$ is also the root of the polynomial $P_1$, and $x_{1,i}\in S(\overline{\alpha_i})$. Hence, from the estimates (\ref{eq5_2}) and Lemma \ref{lm5} we have 
\[
|P'(\alpha_i)|\leq|a_n||\overline{\alpha_i}-\alpha_i|\leq c_{24}\cdot Q^{-\frac{n-1}{2}}.
\]
On the other hand, a Taylor expansion of the polynomial $P_1$ in the interval $S(\alpha_i)$ implies that 
\[
|P'(\alpha_i)|\ge \textstyle\frac12\delta_n\cdot Q.
\]
These two inequalities contradict each other. 

Let us choose a maximal system of algebraic points $\Gamma = \{\boldsymbol{\gamma}_1,\ldots,\boldsymbol{\gamma}_t\}\subset\A_n^2(Q)$ satisfying the condition that rectangles $\sigma(\boldsymbol{\gamma}_k)=\{|x_i-\gamma_{k,i}|<n\delta_n^{-1}Q^{-\frac{n+1}{2}},i=1,2\}$, $1\leq k\leq t$ do not intersect. Furthermore, let us introduce expanded rectangles
\begin{equation}\label{eq5_3}
\sigma'(\boldsymbol{\gamma}_k)=\left\{|x_i-\gamma_{k,i}|<2nh_n\delta_n^{-1}Q^{-\frac{n+1}{2}},i=1,2\right\},\quad k = \overline{1,t},
\end{equation}
and show that
\begin{equation}\label{eq5_4}
B_2\subset\bigcup_{k=1}^t \sigma'(\boldsymbol{\gamma}_k).
\end{equation}
To prove this fact, we are going to show that for any point $\mathbf{x}_1 \in B_1$ there exists a point $\boldsymbol{\gamma}_k\in\Gamma$ such that $\mathbf{x}_1\in\sigma'(\boldsymbol{\gamma}_k)$. 
Since $\mathbf{x}_1 \in B_1$, there is a point $\boldsymbol{\alpha}$ satisfying the inequalities (\ref{eq5_2}). Thus, either $\boldsymbol{\alpha}\in\Gamma$ and $\mathbf{x}_1\in\sigma'(\boldsymbol{\alpha})$, or there exists a point $\boldsymbol{\gamma}_k\in\Gamma$ satisfying
\[
|\alpha_i-\gamma_{k,i}|\leq nh_n\delta_n^{-1}Q^{-\frac{n+1}{2}},\quad i=1,2, 
\]
which implies that $\mathbf{x}_1\in\sigma'(\boldsymbol{\gamma}_k)$. Hence, from (\ref{eq5_1}),(\ref{eq5_3}) and (\ref{eq5_4}) we have:
\[
\textstyle\frac34\cdot \mu_2\Pi \leq\mu_2 B_1\leq \sum\limits_{k=1}^t{\mu_2\sigma_1(\boldsymbol{\gamma}_k)}\leq t\cdot 2^6n^2h_n^2\delta_n^{-2}Q^{-n-1},
\]
which yields the estimate
\[
\#\A_n^2(Q,\Pi)\ge t \ge c_{13}\cdot Q^{n+1}\mu_2\Pi.
\]

\section{Proof of Theorem \ref{main}}

Now we can prove Theorem \ref{main}, which is the main result of the paper.
Consider a set $L_{\varphi}(Q,\gamma,J):=\left\{\mathbf{x}\in\R^2: x_1\in J, \left|\varphi(x_1)-x_2\right|<с_1Q^{-\gamma}\right\}$. Clearly, $M^n_{\varphi}(Q,\gamma,J)= L_{\varphi}(Q,\gamma,J)\cap \A_n^2(Q)$, and our problem is reduced to  estimating the number of algebraic points in the set $\A_n^2(Q)$ lying within the strip $L_{\varphi}(Q,\gamma,J)$. 

\subsection{The lower bound}

The lower bound for $0<\gamma\leq \frac12$ was obtained in the paper \cite{BeGoKu14}, which allows us to consider the case where  $\frac12<\gamma<1$ only. 

Note that the distance between algebraically conjugate numbers is bounded from below, meaning that a certain neighborhood of the line $\varphi_1(x)=x$ must be excluded from consideration. Let us consider the set $D_0 := \left\{x\in J: |\varphi(x)-x|<\textstyle\frac{\varepsilon_1}{2}\right\}$, where $\varepsilon_1 > 0$ is a small positive constant.
Since the number of points $x\in J$ such that $\varphi(x)=x$ is finite, for a sufficiently small constant $\varepsilon_1$ we have that $\mu_1 D_0 <\textstyle\frac14\mu_1 J$. Instead of the interval $J$, let us consider the set $J\setminus D_0 = \bigcup\limits_{k} J_k$, $k\leq c_{5}+1$. The measure of this set  is larger than $\textstyle\frac34\mu_1 J$. 

For every interval $J_k=[b_{k,1},b_{k,2}]$, let us consider a strip $L_{\varphi}(Q,\gamma,J_k)$ and estimate the cardinality of the set $L_{\varphi}(Q,\gamma,J_k)\cap \A_n^2(Q)$. Let us divide the strip $L_{\varphi}(Q,\gamma,J_k)$ into subsets $E_j$ as follows:
\[
E_j :=\left\{\mathbf{x}\in\R^2: x_1\in J_{k,j}, \left|\varphi(x_1)-x_2\right|<c_1Q^{-\gamma}\right\},
\]
where $J_{k,j}=[y_{j-1},y_{j}]$, $y_0=b_{k,1}$ and $y_{j+1}=y_j+c_8Q^{-\gamma}$. The number $t_k$ of subsets $E_j$ can be estimated by the following way:
\begin{equation}\label{eq6_11}
t_k\ge \mu_1 J_k\cdot\left(\mu_1 J_{k,j}\right)^{-1}-1\ge \textstyle\frac12\cdot c_8^{-1}Q^{\gamma}\mu_1 J_k.
\end{equation}
For $\overline{\varphi}_j=\textstyle\frac12\left(\max\limits_{x\in J_{k,j}}{\varphi(x)}+\min\limits_{x\in J_{k,j}}{\varphi(x)}\right)$ consider the squares defined as 
\[
\Pi_j :=\left\{\mathbf{x}\in\R^2: x_1\in J_{k,j}, \left|\overline{\varphi}_j-x_2\right|<\textstyle\frac12c_8Q^{-\gamma}\right\}.
\]
Since the function $\varphi$ is continuously differentiable on the interval $J$, and $\max\limits_{x\in J}{|\varphi'(x)|}<c_{6}$, we get by the mean value theorem that
\[
\left|\max\limits_{x\in J_{k,j}}{\varphi(x)}-\min\limits_{x\in J_{k,j}}{\varphi(x)}\right|<c_{6}\cdot c_8Q^{-\gamma},
\]
which implies that the square $\Pi_j$ is contained in a subset $E_j$. Thus, every set $E_j$ defines the respective square $\Pi_j=I_{j,1}\times I_{j,2}$ of size $\mu_2\Pi_j = c_8^2Q^{-2\gamma}$.

Let us estimate the number of {\it $\left(\textstyle\frac12, \textstyle\frac12\right)$-special} squares $\Pi_j$. To obtain this estimate, let us derive an upper bound on the number of squares $\Pi_j$ satisfying the {\it $\left(l,\frac12,\frac12\right)$-condition} for every $1\leq l\leq L+2$.

For polynomials $P\in\Pl_2(Q)$ of form $P(t)=a_2t^2+a_1t+a_0$, satisfying the conditions
\begin{equation}\label{eq6_0}
\delta Q^{\lambda_{l+1}}\leq|a_2|<\delta Q^{\lambda_l},\quad |P(x_i)|<h\cdot Q^{-\frac12},\quad i=1,2,
\end{equation}
denote by $\Pl_2(Q,l,D)$ a subclass of polynomials $P\in\Pl_2(Q)$ satisfying the inequalities (\ref{eq6_0}) at some point $\mathbf{x}\in D\subset \R^2$. By the definition, if a square $\Pi_j$ satisfies the {\it $\left(l,\frac12,\frac12\right)$-condition}, then the following inequality holds:
\[
\# \Pl_2(Q,l,\Pi_j) \leq \delta^3\cdot 2^{l+3}Q^{1+2\lambda_{l+1}}\mu_2\Pi_j.
\]
Consider the expanded sets $E_s = \bigcup\limits_{i=j_s}^{j_s+T(l)}{E_i}$ composed of $T(l)$ subsets $E_j$, where
\begin{equation}\label{eq6_9}
T(l)=c_{24} Q^{\gamma-\lambda_l}, \qquad 
c_{24}=\textstyle\frac{1}{8}\cdot\delta^{-1}c_8^{-1}\left(|d_1|+|d_2|+\varepsilon_1\right)^{-1}\cdot\min\left\{c_{6},\varepsilon_1^{-1}\right\},
\end{equation}
and $j_1=1$, $j_{s+1}=j_s+T(l)+1$. By the inequality (\ref{eq6_11}), the number of expanded sets can be estimated as follows:
\[
s\leq t_k\cdot T(l)^{-1}\leq c_8T(l)^{-1}Q^{\gamma}\mu_1 J_k.
\]

Now let us show that at least $\left(1-2^{-l-3}\right)\cdot T(l)$ squares $\Pi_j\subset E_s$ satisfy the {\it $\left(l,\frac12,\frac12\right)$-condition}.
By definition of the set $E_s$, for every point $\mathbf{x}\in E_s$ we obtain:
\begin{equation}\label{eq6_2}
x_1\in I_1,\qquad\mu_1 I_1=c_8\cdot c_{24} Q^{-\lambda_l}.
\end{equation}
On the other hand, since $\varphi$ is continuously differentiable on the interval $J$ and $\max\limits_{x\in J}{|\varphi'(x)|}<c_{6}$ then $E_s\subset \Pi$, where $\Pi = I_1\times I_2$ and $\mu_1 I_2 = c_{6} \mu_1 I_1$. Thus $\#\Pl_2(Q,l,E_s)\leq \#\Pl_2(Q,l,\Pi)$, and we only need to estimate the quantity $\#\Pl_2(Q,l,\Pi)$.

By the third inequality of Lemma \ref{lm1}, for every polynomial $P\in\Pl_2(Q,l,\Pi)$ satisfying the system (\ref{eq6_0}) at a point $\mathbf{x}_0\in \Pi$, the inequalities
\begin{equation}\label{eq6_3}
|x_{0,i}-\alpha_i|<\left(|P(x_{0,i})|\cdot|a_2|^{-1}\right)^{-\frac12} < h^{\frac12}\cdot Q^{-\frac14} < \textstyle\frac{\varepsilon_1}{8},\\
\end{equation}
are satisfied for $Q>Q_0$ and $x_{0,i}\in S(\alpha_i)$. From (\ref{eq6_3}) and the condition $|x_1-x_2|>\varepsilon_1$, we obtain the following lower bound for $|P'(\alpha_i)|$:
\begin{equation}\label{eq6_5}
|P'(\alpha_i)|=|a_2|\cdot|\alpha_1-\alpha_2|>\textstyle\frac{3}{4}\cdot \varepsilon_1\cdot|a_2|.
\end{equation}
Moreover, from the inequalities (\ref{eq6_3}) we have
\begin{equation}\label{eq6_6}
|P'(x_{0,i})|\leq|a_2|\cdot\left(|\alpha_1-x_{0,i}|+|\alpha_2-x_{0,i}|\right)\leq \left(|d_1|+|d_2|+\textstyle\frac12\varepsilon_1\right)\cdot|a_2|,
\end{equation}
where $\mathbf{d}$ is a midpoint of the rectangle $\Pi$. Let us estimate the polynomials $P\in\Pl_2(Q,l,\Pi)$ at a point $\mathbf{d}\in\Pi$. From a Taylor expansion of the polynomial $P$ in the interval $I_i$ and inequalities (\ref{eq6_0}), (\ref{eq6_6}) we have:
\begin{equation}\label{eq6_6}
|P(d_i)|<\left(|d_1|+|d_2|+\varepsilon_1\right)\cdot|a_2|\cdot\mu_1 I_i.
\end{equation}

Fix a number $a$ and consider a subclass of polynomials $P$ with the same leading coefficient:
\[
\Pl_2(Q,l,\Pi,a):=\{P\in\Pl_2(Q,l,\Pi):a_2=a\}.
\]
It is clear that the inequality $\#\Pl_2(Q,l,\Pi,a) > 0$ holds only if the conditions (\ref{eq6_0}) are satisfied. Hence, the number of classes under consideration can be estimated as follows:
\begin{equation}\label{eq6_7}
\#\{a\}\leq \delta Q^{\lambda_l}.
\end{equation}

Now let us estimate the number of polynomials in a subclass $\Pl_2(Q,l,\Pi,a)$. Choose a polynomial $P_0\in\Pl_2(Q,l,\Pi,a)$ and consider the differences between polynomials $P_0$ and $P_j\in\Pl_2(Q,l,\Pi,a)$ at a point $\mathbf{d}$. From the estimates (\ref{eq6_6}) it follows that
\[
|P_0(d_i)-P_j(x_{0,i})|=|(a_{0,1}-a_{j,1})d_i+(a_{0,0}-a_{j,0})|\leq 2c_{25}\cdot |a|\cdot\mu_1 I_i,
\]
where $c_{25}=|d_1|+|d_2|+\varepsilon_1$. Thus, the number of different polynomials $P_j\in\Pl_2(Q,l,\Pi,a)$ does not exceed the number of integer solutions of the following system:
\[
|b_1d_i+b_0|\leq 2c_{25}\cdot |a|\cdot\mu_1 I_i,\quad i=1,2.
\]
Let us apply Lemma \ref{lm6} with $K_i=2c_{25}\cdot |a|\cdot\mu_1 I_i$. From the estimates (\ref{eq6_0}) and (\ref{eq6_2}), we can easily verify that $4\varepsilon_1^{-1} K_1 < 1$ and $4K_2 < 1$, which leads to the inequality
\begin{equation}\label{eq6_12}
\#\Pl_2(Q,l,\Pi,a)\leq 1.
\end{equation}
Hence, from  the inequality (\ref{eq6_7}) we obtain the estimate
\begin{equation}\label{eq6_8}
\#\Pl_2(Q,l,\Pi)= \sum\limits_{a}\#\Pl_2(Q,l,\Pi,a) \leq \delta Q^{\lambda_l}.
\end{equation}

Let us consider the case where $1\leq l\leq L+1$. Assume that the inequality 
\begin{equation}\label{eq6_10}
\#\Pl_2(Q,l,\Pi_j) > \delta^3\cdot 2^{l+3}Q^{1+2\lambda_{l+1}}\mu_2\Pi_j
\end{equation}
holds for $2^{-l-3}\cdot T(l)$ squares $\Pi_j$. By Lemma \ref{lm1}, for a polynomial $P\in\Pl_2(Q)$ the set of points $\mathbf{x}$ satisfying (\ref{eq6_0}) is contained in the following set:
\[
\sigma_P:=\left\{|x_i-\alpha_i|\leq hQ^{-\frac12}\cdot|P'(\alpha_i)|^{-1}, x_i\in S(\alpha_i), i=1,2\right\}.
\]
From the inequalities (\ref{eq6_0}) and (\ref{eq6_5}) it is easy to see that the measure of the set $\sigma_P$ is at most a half of the size of $\Pi_j$ for $1\leq l\leq L+1$ and $c_8 > h\delta^{-1}\varepsilon_1^{-1}$.
Therefore no polynomial $P\in\Pl_2(Q)$ satisfies the inequalities (\ref{eq6_0}) at three points that lie inside three different squares $\Pi_j$.

Since $\Pi_j\subset E_j\subset E\subset \Pi$ we have $\bigcup\limits_j{\Pi_j}\subset \Pi$. Then by our assumption and the inequality $\#\Pl_2(Q,l,\Pi_j)\ge 0$ we get:
\[
\#\Pl_2(Q,l,\Pi)\ge \sum\limits_{i=j_s}^{j_s+T(l)}{\#\Pl_2(Q,l,\Pi_i)} \ge\textstyle\frac{1}{2^{l+3}}\cdot T(l)\cdot \#\Pl_2(Q,l,\Pi_j).
\]
From the inequalities (\ref{eq6_9}) and (\ref{eq6_10}) for $1\leq l\leq L$, we obtain:
\[
\#\Pl_2(Q,l,\Pi)\ge c_{24}\delta^3\cdot c_8^2\cdot Q^{1-\gamma+2\lambda_{l+1}-\lambda_l}>\delta Q^{\lambda_l},
\]
for $c_8 > 8\delta^{-1}c_{25}\cdot\left(\min\left\{c_{6}, \varepsilon_1^{-1}\right\}\right)^{-1}$. This inequality  contradicts the estimate (\ref{eq6_8}). 

For $l =L +1$ we can use the inequalities (\ref{eq6_9}) and (\ref{eq6_10}) to obtain:
\[
\#\Pl_2(Q,l,\Pi)\ge c_{24}\delta^3\cdot c_8^2\cdot  Q^{\gamma-\lambda_{L+1}}> \delta Q^{\gamma-1+\frac{1-\gamma}{2}\cdot \left[\frac{3-2\gamma}{1-\gamma}\right]}\ge \delta Q^{\gamma-1+\frac{3-2\gamma}{2}-\frac{1-\gamma}{2}}>\delta Q^{\frac{\gamma}{2}},
\]
for $c_8 > 8\delta^{-1}c_{25}\cdot\left(\min\left\{c_{6}, \varepsilon_1^{-1}\right\}\right)^{-1}$. On the other hand, the estimates (\ref{eq6_8}) imply that:
\[
\#\Pl_2(Q,l,\Pi)\leq \delta Q^{\lambda_{L+1}}=\delta Q^{1-\frac{1-\gamma}{2}\cdot \left[\frac{3-2\gamma}{1-\gamma}\right]}\leq \delta Q^{1-\frac{3-2\gamma}{2}}= \delta Q^{\gamma-\frac12} < \delta Q^{\frac{\gamma}{2}}
\]
for $\gamma < 1$, which contradicts the previous inequality.

This argument proves that the number of squares $\Pi_j \subset E_s$, satisfying the {\it $\left(l,\frac12, \frac12\right)$-condition} for $1\leq l\leq L+1$ is larger than $\left(1-2^{-l-3}\right)\cdot T(l)$. 

The case where $l = L+2$ needs to be treated differently. From Lemma \ref{lm1} and the inequalities (\ref{eq6_5}), it follows that the set of points $\mathbf{x}$ satisfying the inequalities (\ref{eq6_0}) for some polynomial $P$ is contained in the following set:
\[
\sigma_P:=\left\{|x_i-\alpha_i|\leq h\varepsilon_1^{-1}\cdot Q^{-\frac12}\cdot|a_2|^{-1}, i=1,2\right\}.
\]
and the measure of the set $\sigma_P$ is larger than the size of the square $\Pi_j$. It means that a single polynomial can belong to a large number of different sets $\Pl_2(Q,l,\Pi_j)$. Let us estimate this number for a fixed polynomial $P\in \Pl_2(Q,l,\Pi)$. Since the side of the square $\sigma_P$ is larger than the width of the strip $L_{\varphi}(Q,\gamma,J_k)$, we have:
\[
\#\{\Pi_j: P\in\Pl_2(Q,l,\Pi_j)\}\leq 2h\varepsilon_1^{-1}c_8^{-1}\cdot Q^{\gamma-\frac12}\cdot|a_2|^{-1}.
\]
Now from inequalities (\ref{eq6_8}) and the estimates (\ref{eq6_12}), we can obtain that
\[
\begin{split}
\#\bigcup\limits_{P\in\Pl_2(Q,l,\Pi)}\{\Pi_j: P\in\Pl_2(Q,l,\Pi_j)\}&\leq 2h\varepsilon_1^{-1}c_8^{-1}\cdot Q^{\gamma-\frac12}\sum\limits_{1\leq|a_2|<\delta Q^{\gamma-\frac12}}{|a_2|^{-1}}\leq \\
&\leq  2^4\varepsilon_1^{-1}hc_8^{-1}\left(\gamma-\textstyle\frac12\right)Q^{\gamma-\frac12}\ln Q <\textstyle\frac{1}{2^{l+3}}\cdot T(l),
\end{split}
\]
for $\gamma < 1$ and $Q>Q_0$.

This implies that the inequality $\#\Pl_2(Q,l,\Pi)>0$ can only be satisfied for $2^{-l-3}\cdot T(l)$ squares $\Pi_j\subset E_s$ and, therefore the number of squares $\Pi_j \subset E_s$, satisfying the {\it $\left(l,\frac12, \frac12\right)$-condition} for $l= L+2$ is larger than $\left(1-2^{-l-3}\right)\cdot T(l)$.

Now inequality (\ref{eq6_11}) yields that the number of squares $\Pi_j\in L_{\varphi}(Q,\gamma,J_k)$ satisfying the {\it $\left(l,\frac12,\frac12\right)$-condition} for $1\leq l\leq L+2$ is bigger than $\left(1-\frac{1}{2^{l+3}}\right)\cdot t_k$. Thus,  we have: 
\[
\sum\limits_{\substack{P_j,l: P_j \text{ satisfy } \\
\left(l,1/2, 1/2\right)\text{-condition}}}{1}\ge\sum\limits_{l=1}^{L+2}{\left(1-\textstyle\frac{1}{2^{l+3}}\right)\cdot t_k}= \left(L+2-\textstyle\frac14+\textstyle\frac{1}{2^{L+3}}\right) \cdot t_k > \left(L+\textstyle\frac74\right) \cdot t_k.
\]
Assume that the number of squares $\Pi_j\subset L_{\varphi}(Q,\gamma,J_k)$ which satisfy the {\it $\left(l,\frac12,\frac12\right)$-condition} for all $1\leq l\leq L+2$ is smaller than $\textstyle\frac34\cdot t_k$. Then we have
\[
\sum\limits_{\substack{P_j,l: P_j \text{ satisfy } \\
\left(l,1/2, 1/2\right)\text{-condition}}}{1}\leq \textstyle\frac34\cdot t_k\cdot (L+2)+ \textstyle\frac14\cdot t_k\cdot (L+1) = \left(L+\textstyle\frac74\right) \cdot t_k,
\]
which contradicts the previous estimate. Thus, there exist at least $\textstyle\frac34\cdot t_k$  {\it $\left(\textstyle\frac12, \textstyle\frac12\right)$-special} squares $\Pi_j\subset L_{\varphi}(Q,\gamma,J_k)$. These squares satisfy the conditions of Theorem \ref{th2}, allowing us to write the following estimate:
\[
\# \A_n^2(Q,\Pi_j) \ge c_{13}Q^{n+1}\mu_2\Pi_j = c_{13}c_8^2\cdot Q^{n+1-2\gamma}.
\]
The inequality (\ref{eq6_11}) and the upper bound on the number of {\it $\left(\textstyle\frac12, \textstyle\frac12\right)$-special} squares imply that
\[
\#\left(L_{\varphi}(Q,\gamma,J_k)\cap \A_n^2(Q)\right) \ge \textstyle\frac34 c_{13}c_8^2\cdot t_k\cdot Q^{n+1-2\gamma} \ge \textstyle\frac38 c_{13}c_8\cdot Q^{n+1-\gamma}\mu_1 J_k.
\]

These inequalities, in turn,  lead us to the following lower bound on $\# M_{\varphi}(Q,J,\gamma)$:
\[
\# M_{\varphi}(Q,J,\gamma) \ge \textstyle\frac38 c_{13}c_8\cdot Q^{n+1-\gamma} \sum\limits_k\mu_1 J_k \ge \textstyle\frac{9}{32} c_{13}c_8\cdot \mu_1 J\cdot  Q^{n+1-\gamma} = c_{2}\cdot Q^{n+1-\gamma}.
\]

\subsection{Upper bound}

As in the previous section, let us divide the set $L_{\varphi}(Q,\gamma,J)$, $J=[b_1,b_2]$ into subsets $E_j$, $1\leq j\leq t$:
\[
E_j :=\left\{\mathbf{x}\in\R^2: x_1\in J_{j}, \left|\varphi(x_1)-x_2\right|<\left(\textstyle\frac12+c_{6}\right)\cdot c_8Q^{-\gamma}\right\},
\]
where $J_{j}=[y_{j-1},y_{j}]$, $y_0=b_{1}$ и $y_{j+1}=y_j+\left(\textstyle\frac12+\textstyle\frac32c_{6}\right)\cdot c_8Q^{-\gamma}$ and the number of subsets $E_j$ satisfies the inequality
\begin{equation}\label{eq6_13}
t\leq \mu_1 J\cdot\left(\mu_1 J_{j}\right)^{-1}\leq \left(\textstyle\frac12+\textstyle\frac32c_{6}\right)^{-1}\cdot c_8^{-1}Q^{\gamma}\mu_1 J.
\end{equation}

Once again for $\overline{\varphi}_j=\textstyle\frac12\left(\max\limits_{x\in J_{j}}{\varphi(x)}+\min\limits_{x\in J_{j}}{\varphi(x)}\right)$ let us consider the squares
\[
\Pi_j :=\left\{\mathbf{x}\in\R^2: x_1\in J_{j}, \left|\overline{\varphi}_j-x_2\right|<\left(\textstyle\frac12+\textstyle\frac32c_{6}\right)\cdot c_8Q^{-\gamma}\right\}.
\]
Since the function $\varphi$ is continuously differentiable on the interval $J$, and $\max\limits_{x\in J}{|\varphi'(x)|}<c_{6}$, it is easy to see that each subset $E_j$ is contained in the respective square $\Pi_j$: $E_j\subset \Pi_j$, $1\leq j\leq t$. 

Note that the squares $\Pi_j$ satisfy the conditions of Theorem \ref{th1}. Therefore, we have
\[
\# \A_n^2(Q,\Pi_j) \leq c_{12}Q^{n+1}\mu_2\Pi_j = c_{12}c_8^2\left(\textstyle\frac12+\textstyle\frac32c_{6}\right)^2\cdot Q^{n+1-2\gamma}.
\]
These inequalities, together with the estimate (\ref{eq6_13}), lead to  the following upper bound for $\# M_{\varphi}(Q,I,\gamma)$:
\[
\# M_{\varphi}(Q,J,\gamma) \leq \sum\limits_{j=1}^t \# \A_n^2(Q,\Pi_j)\leq c_{12}c_8\left(\textstyle\frac12+\textstyle\frac32c_{3}\right)\cdot \mu_1 J\cdot Q^{n+1-\gamma} = c_{3}\cdot Q^{n+1-\gamma}.
\]

\newpage

УДК 511.42

Распределение точек с алгебраически сопряженными координатами в окрестности гладких кривых

Key words and phrases: algebraic numbers, metric theory of Diophantine approximation, Lebesgue measure.

Ключевые слова: алгебраические числа, метрическая теория диофантовых приближений, мера Лебега.

Аннотация.

Пусть $\varphi:\R\rightarrow \R$ непрерывно дифференциируемая на интервале $J\subset\R$ функция и пусть $\boldsymbol{\alpha}=(\alpha_1,\alpha_2)$ точка с алгебраически сопряженными координатами, минимальный многочлен $P$ которых является многочленом степени $\leq n$ и высоты $\leq Q$. Определим через $M^n_\varphi(Q,\gamma, J)$ множество точек $\boldsymbol{\alpha}$, удовлетворяющих условию $|\varphi(\alpha_1)-\alpha_2|\leq c_1 Q^{-\gamma}$. В работе доказано, что для любого действительного $0<\gamma<1$ и достаточно большого $Q$ существуют положительные величины $c_2<c_3$, не зависят от $Q$, для которых выполняются оценки $c_2\cdot Q^{n+1-\gamma}<\# M^n_\varphi(Q,\gamma, J)< c_3\cdot Q^{n+1-\gamma}$.

\newpage

\bigskip

{\noindent Берник Василий Иванович}\\
{\small
{Институт математики НАН Беларуси,\\
ул. Сурганова, 11, 220072, Минск, Беларусь}\\
E-mail: bernik@im.bas-net.by
}

\bigskip

{\noindent Гётце Фридрих}\\
{\small
{Department of Mathematics, University of Bielefeld,\\
Postfach 100131, 33501, Bielefeld, Germany}\\
E-mail: goetze@math.uni-bielefeld.de

\bigskip

{\noindent Гусакова Анна Григорьевна}\\
{\small
{Институт математики НАН Беларуси,\\
ул. Сурганова, 11, 220072, Минск, Беларусь}\\
E-mail: gusakova.anna.0@gmail.com
}

\end{document}